\documentclass[11pt]{amsart}
\usepackage[parfill]{parskip}    
\usepackage{fullpage}
\usepackage{graphicx}
\usepackage{amssymb,amsmath,latexsym}
\usepackage{tikz}
\usetikzlibrary{arrows,decorations.markings,decorations.pathreplacing,calc,matrix,intersections}
\tikzset{->-/.style={decoration={markings,mark=at position #1 with {\arrow{>}}},postaction={decorate}}}

\usepackage{epstopdf}
\usepackage{comment}
\usepackage{color}
\usepackage{float}
\usepackage{amscd,amsmath}
\usepackage{amssymb,amsmath,latexsym,color,enumerate,tikz}
\usepackage[all]{xy}
\usepackage[varg]{pxfonts}
\usepackage{amsmath,amsthm,amsfonts,amssymb,fancyhdr,graphics,relsize,tikz-cd,mathtools,faktor,tikz,changepage}
\usepackage{graphicx}
\usepackage{placeins}
\usepackage{mathabx}

\usepackage{dsfont}
\definecolor{red}{rgb}{1,0,0} 

 \definecolor{darkgreen}{rgb}{0, .7, 0}

 \definecolor{purple}{rgb}{.7, 0, 1}

\newcommand{\Z}{{\mathbb{Z}}}

\newcommand{\C}{{\mathcal{C}}}

\newcommand{\link}{{\mbox{Link}}}

\makeatletter
\newsavebox{\@brx}
\newcommand{\llangle}[1][]{\savebox{\@brx}{\(\m@th{#1\langle}\)}%
  \mathopen{\copy\@brx\mkern2mu\kern-0.9\wd\@brx\usebox{\@brx}}}
\newcommand{\rrangle}[1][]{\savebox{\@brx}{\(\m@th{#1\rangle}\)}%
  \mathclose{\copy\@brx\mkern2mu\kern-0.9\wd\@brx\usebox{\@brx}}}
\makeatother

\tikzset{mynode/.style={draw,circle,fill=black,inner sep=2pt,outer sep=0.5pt}}

\newtheorem{theorem}{Theorem}[section]
\newtheorem*{theorem*}{Theorem}
\newtheorem*{cor*}{Corollary}

\newtheorem*{thm:centralisers}{Theorem \ref{thm: centralisers structure}}

\newtheorem*{thm:abelian_splittings}{Theorem \ref{thm: abelian splitting of RAAGs}}

\newtheorem*{thm:direct_product}{Theorem \ref{thm: directly decomposable}}

\newtheorem*{thm:solvable_sbgps}{Theorem \ref{thm: solvable subgroups are metabelian}}

\newtheorem*{thm:two_gen}{Theorem \ref{thm: two generated subgroups}}

\newtheorem*{lemma*}{Lemma}

\newtheorem{lemma}[theorem]{Lemma}
\newtheorem{corollary}[theorem]{Corollary}

\theoremstyle{definition}
\newtheorem{definition}[theorem]{Definition}
\newtheorem{example}[theorem]{Example}

\theoremstyle{remark}
\newtheorem{remark}[theorem]{Remark}

\usepackage{hyperref}

\title{Pro-$\C$ RAAGs}
\begin{document}

\author{Montserrat Casals-Ruiz}\address{Ikerbasque -- Basque Foundation for Science and Matematika Saila,  UPV/EHU,  Sarriena s/n, 48940, Leioa -- Bizkaia, Spain}\email{montsecasals@gmail.com}

\author{Matteo Pintonello}\address{Matematika Saila,  UPV/EHU,  Sarriena s/n, 48940, Leioa -- Bizkaia, Spain}\email{mattepint@gmail.com}

\author{Pavel Zalesskii}\address{Departamento de Matem\'atica, Universidade de Bras\'ilia, Brasil}\email{pz@mat.unb.br}
\thanks{
The first two authors were supported by the Basque Government Grant IT1483-22 and the Spanish Government grants PID2019-107444GA-I00 and PID2020-117281GB-I00. The second author was supported by grant FPI-2018 of the Spanish Governement. The third author was supported by CNPq.}

\begin{abstract}
Let $\mathcal{C}$ be a class of finite groups closed under taking subgroups, quotients, and extensions with abelian kernel. The right-angled Artin pro-$\mathcal{C}$ group $G_\Gamma$ (pro-$\C$ RAAG for short) is the pro-$\mathcal{C}$ completion of the right-angled Artin group $G(\Gamma)$ associated with the finite simplicial graph $\Gamma$. 

In the first part, we describe structural properties of pro-$\mathcal{C}$ RAAGs. Among others, we describe the centraliser of an element and show that pro-$\C$ RAAGs satisfy the Tits' alternative, that standard subgroups are isolated, and that 2-generated pro-$p$ subgroups of pro-$\C$ RAAGs are either free pro-$p$ or free abelian pro-$p$. 

In the second part, we characterise splittings of pro-$\C$ RAAGs in terms of  the defining graph. More precisely, we prove that a pro-$\mathcal{C}$ RAAG $G_\Gamma$ splits as a non-trivial direct product if and only if $\Gamma$ is a join and it splits over an  abelian pro-$\mathcal{C}$ group if and only if a connected component of $\Gamma$ is a complete graph or it has a complete disconnecting subgraph. We then use this characterisation to describe an abelian JSJ decomposition of a pro-$\C$ RAAG, in the sense of Guirardel and Levitt \cite{Guirardel}. 
\end{abstract}
\maketitle

\section{Introduction}

Given a finite simplicial graph $\Gamma$, the associated right-angled Artin group (abbreviated RAAG) $G_\Gamma$ is the
the group given by the presentation:

\begin{equation}\label{eq:presentaion}
G(\Gamma) = \langle V(\Gamma) \mid [u, v] = 1 \hbox{ whenever } (u, v) \in E(\Gamma)\rangle.
\end{equation}

This class of groups has been widely studied in geometric group theory on account of their intrinsically rich structure, and their natural appearance in several branches of computer science and mathematics. Crucial examples, that shaped the theory of presentations of groups,
arise from the study of their subgroups, notably Bestvina and Brady’s example of a group that is homologically finite (of type FP) but not geometrically finite (in fact not of type $F_2$); and Mihailova’s example of a group with unsolvable subgroup membership problem. In recent years, using results of Wise, Haglund, and others, many important families of groups have been shown to be virtually virtual retracts of RAAGs, among others, hyperbolic 3-manifold groups, which was one of the key ingredients in Agol’s proof of the virtual Haken conjecture. Despite the abundance of subgroups, there are some structural results on subgroups, namely RAAGs satisfy the strong Tits alternative and in fact, 2-generated subgroups are either free or free abelian.

Many algebraic properties of RAAGs can be characterised in terms of the properties of the defining graph $\Gamma$. For instance, a RAAG $G(\Gamma)$ splits as a non-trivial free product if and only if $\Gamma$ is not connected; and $G(\Gamma)$ splits as a non-trivial direct product if and only if $\Gamma$ is a join. More interestingly, the splittings of $G(\Gamma)$ over an abelian group are also encoded by the graph $\Gamma$, namely, $G(\Gamma)$ splits over an abelian group if and only if either $\Gamma$ is disconnected, or it is a complete graph or it has a disconnecting complete subgraph (see \cite{Clay} and \cite{Groves}). One can further use the natural splittings determined by complete disconnecting subgraphs of $\Gamma$ and refine them to obtain a JSJ-decomposition for $G(\Gamma)$. An overview of RAAGs can be found in \cite{Charney, Koberda}.

RAAGs are linear and have excellent residual properties thus one can consider the pro-$\mathcal{C}$ completion $G_\Gamma$ of a RAAG $G(\Gamma)$, called the \emph{pro-$\mathcal{C}$ right-angled Artin group}. In other words, the pro-$\mathcal{C}$ group $G_\Gamma$ is given by the pro-$\mathcal{C}$ presentation given in \eqref{eq:presentaion}. In fact, in \cite{Wilkes} it is shown that RAAGs are distinguished from each other by their pro-$p$ completions for any choice of prime $p$, i.e. the pro-$p$ RAAGs $G_\Gamma$ and $G_\Delta$ are isomorphic if and only if so are the RAAGs $G(\Gamma)$ and $G(\Delta)$. 

In this article, we study pro-$\mathcal{C}$ RAAGs, where $\C$ is a class of finite groups closed under taking subgroups, quotients, and extensions. If $\C$ is the class of all finite, finite $p$ or finite soluble (etc.) groups we shall use the terms profinite, pro-$p$, prosoluble (etc.) RAAG. Note that the subgroup structure of a profinite RAAG is also very rich since the profinite completion of virtual retracts (and so the profinite completion of hyperbolic 3-manifold groups, limit groups, one relator groups with torsion) is virtually a subgroup of the profinite completion of a pro-$\C$ RAAG.

We begin by studying some structural results on subgroups of RAAGs. More precisely, we describe subgroups of pro-$\C$ RAAGS that do not contain free non-abelian pro-$\C$ subgroups. This applies for example to subgroups that satisfy a law and in particular to soluble subgroups of  pro-$\C$ RAAGS. In fact, we show that pro-$\C$ RAAGs satisfy the Tits' alternative:

\begin{thm:solvable_sbgps}[Tits' alternative]
Let $H$ be a closed subgroup of a pro-$\mathcal{C}$ RAAG $G_\Gamma$ that does not contain a free non-abelian pro-$p$ subgroup for any $p$. Then $H$ is metabelian and polycyclic. Moreover, if $H$ is pro-$p$, then $H$ is abelian.
\end{thm:solvable_sbgps}

As a hyperbolic virtually compact special group is a virtual retract of a RAAG, Theorem \ref{thm: solvable subgroups are metabelian} can be considered as a generalization of Theorems D and E of \cite{WilZal17}.  

In the next theorem we  describe 2-generated pro-$p$ subgroups of $G_\Gamma$.

\begin{thm:two_gen}
Let $H$ be a two-generated  pro-$p$ subgroup of a pro-$\C$ RAAG $G_\Gamma$. Then $H$ is either free pro-$p$ or free abelian.
\end{thm:two_gen}

We next turn our attention to an important class of subgroups, centralisers of elements, and give the following description:

\begin{thm:centralisers}
Let $G=G_\Gamma$ be a pro-$\C$ RAAG and let $g_0\in G$. Then there is an element $g$ in the conjugacy class of $g_0$ such that its centraliser is of the form

$$C_G(g)=H_1 \times \cdots \times H_s \times \overline{\langle \link(g)\rangle }$$
where:
\begin{enumerate}
    \item $\alpha(H_i),~ \alpha(H_j),~\link(g)$ are all disjoint for $i\neq j$;
    \item $G_{\alpha(g)}=G_{\alpha(H_1)}\times \cdots \times G_{\alpha(H_s)}$
    \item $H_i$ are projective  pro-$\C$ groups with center;
    \item if $G$ is pro-$p$, $H_i=\overline{\langle h_i \rangle}$ and $g= h_1^{k_1} \cdots h_s^{k_s}$, for some $k_i\in \mathbb Z_p$.
\end{enumerate}
\end{thm:centralisers}

As in the abstract case, some of the algebraic properties of the pro-$\mathcal{C}$ RAAG $G_\Gamma$ can be characterised in terms of properties of the defining graph $\Gamma$. In \cite{Wilkes}, it is proven that the profinite RAAG $G_\Gamma$ splits as a non-trivial profinite free product if and only if $\Gamma$ is disconnected. In \cite{SnopceZalesskii}, Snopce and Zalesskii proved that the pro-$p$ RAAGs that are Block-Kato pro-$p$ groups are precisely the pro-$p$ RAAGs for which each closed subgroup is itself a right-angled Artin pro-$p$
group (possibly infinitely generated) and this class coincides with the pro-$p$ RAAGs whose defining graph does not contain squares and lines with four vertices as induced subgraphs. Kochloukova and the third author \cite[Theorems C and D]{KZ23} studied the finite generation of a normal subgroup pro-$p$ RAAG in terms of the relation between the defining graph $\Gamma$ and the natural quotient map; in fact finitely generated normal subgroups of pro-$\C$ RAAG were described in \cite[Theorem A]{KZ23} in terms of the quotients (the abstract version of the result was proved by J.Lopez de Gamiz Zearra and the first author in \cite{CRLdG22}). 

In this context, we show that one can characterise when a pro-$\mathcal{C}$ RAAG splits as a non-trivial direct product in terms of the defining graph, namely, we prove the following:

\begin{thm:direct_product}
Let $G_\Gamma$ be a pro-$\mathcal{C}$ RAAG. Then $G_\Gamma$ decomposes as a non-trivial direct product if and only if $\Gamma$ is a join.
\end{thm:direct_product}

We next study abelian splittings of pro-$\mathcal{C}$ RAAGs and characterise the existence in terms of the defining graph. In analogy to the abstract case, see \cite{Groves}, we prove the following:

\begin{thm:abelian_splittings}
Let $G=G_\Gamma$ be a pro-$\mathcal{C}$ RAAG associated with a connected graph $\Gamma$. Then $G$ splits over an abelian group if and only if either $\Gamma$ is a complete graph or $\Gamma$ has a disconnecting complete graph.
\end{thm:abelian_splittings}

We then describe an abelian JSJ decomposition of a pro-$\C$ RAAGs (in the sense of Guirardel and Levitt \cite{Guirardel}). This decomposition is given explicitly in terms of the defining graph, see Theorem \ref{thm: A JSJ decomposition}.

Bass-Serre theory of groups acting on trees is one of the main tools for proving the splitting results in the abstract case. In our case, we use a profinite analog of this theory that has been developed mainly by Mel'nikov, Ribes, and Zalesskii that the reader can find in \cite{RibesBook} (see also \cite{NewHorizons} for the pro-$p$ version of it).

\section{Profinite groups acting on profinite trees}

In this section, we describe the analogue results of Bass-Serre theory for profinite groups acting on profinite trees. A deeper description can be found in \cite{RibesBook} or, for pro-$p$ groups, in \cite{RZBook} and \cite{NewHorizons}.

Throughout the article, we assume $\mathcal{C}$ to be a class of finite groups closed under taking subgroups, homomorphic images, direct product, and extensions with abelian kernel. The \emph{primes involved in} $\C$ is the set of primes that divide the order of a group $G\in \C$ and are denoted by $\pi(\C)$. As usual, if $\pi$ is a set of primes, a pro-$\pi$ group is the inverse limit of finite groups of $\pi$-order.

\begin{definition}[Profinite graph]
A profinite graph is a profinite space $\Gamma$ with a distinguished non-empty closed subset $V(\Gamma)$ and two continuous maps (called \emph{incidence maps}) $d_0,d_1:\Gamma\to V(\Gamma)$ which restrict to the identity on $V(\Gamma)$.
\end{definition}
The elements of $V(\Gamma)$ are the \emph{vertices} of the profinite graph, whereas the elements of $E(\Gamma):=\Gamma\smallsetminus V(\Gamma)$ are the \emph{edges}. A \emph{morphism} $\alpha:\Gamma\to \Delta$ is a map of profinite spaces respecting incidence maps, so $\alpha d_i=d_i \alpha$ for $i\in \{0,1\}$. A profinite graph is the inverse limit of its finite quotients graphs (see \cite[Proposition 1.5]{NewHorizons}) and we say that $\Gamma$ is \emph{connected} if all of these finite quotient graphs are connected (as abstract finite graphs).\\
For each profinite graph $\Gamma$, we define $(E^*(\Gamma),*)=(\Gamma/V(\Gamma),*)$ the pointed profinite quotient space, where the distinguished point is the representative of $V(\Gamma)$. For each prime $p$, we have a complex of free profinite $\mathbb{F}_p$-modules
\begin{equation} \label{exactsequence}
   0\longrightarrow \mathbb{F}_p[[E^*(\Gamma),*]] \xrightarrow{\delta} \mathbb{F}_p[[V(\Gamma)]] \xrightarrow{\epsilon} \mathbb{F}_p\longrightarrow 0
\end{equation}
where the maps are defined as $\delta(e)=d_1(e)-d_0(e)$ for each $e\in E^*(\Gamma)$ and $\epsilon(v)=1$ for each $v\in V(\Gamma)$.

\begin{definition}[Profinite tree]
A profinite graph $\Gamma$ is a \emph{pro-$p$ tree} if the associated chain complex (\ref{exactsequence}) is an exact sequence.
A profinite graph is a pro-$\C$ tree if the associated chain complex (\ref{exactsequence}) is an exact sequence for each prime $p\in \pi(\C)$.
\end{definition}

In particular, a pro-$\C$ tree is a pro-$\pi$ tree with $\pi=\pi(\C)$ (see \cite[Proposition 2.4.2]{RibesBook}). As every pro-$\C$ group is a pro-$\pi(\C)$ group, we can often reduce the study of pro-$\C$ groups acting on pro-$\C$ trees to the study of pro-$\pi$ groups acting on pro-$\pi$ trees. For this reason, many theorems we refer to are originally stated in the pro-$\pi$ version in the sources we cite but they are still valid for pro-$\mathcal{C}$ groups and trees, and we state them in this form.

All the subtrees of a pro-$\mathcal{C}$ tree $\Gamma$ are partially ordered by inclusion and the minimal subtree containing two vertices $v,w\in V(\Gamma)$, which we denote as $[v,w]$, is called a \emph{geodesic}.

Some results that are valid for abstract trees are true for pro-$\C$ trees too. For example, we will make use of Helly's Theorem for pro-$\mathcal{C}$ trees.

\begin{lemma}[Helly Property]\label{lem: Helly}
    Let $S=\{T_i, i\in I\}$ be an arbitrary family of non-empty pro-$\mathcal{C}$ subtrees of a pro-$\mathcal{C}$ tree $T$ and suppose $T_i\cap T_j\neq \varnothing$ for each $i,j\in I$. Then $\bigcap_{i\in I} T_i \neq \varnothing$.
\end{lemma}
\begin{proof}
As the pro-$\mathcal{C}$ tree $T$ is compact, it suffices to prove that every finite subset of $S$ has a non-empty intersection, so we prove the result for $S_k=\{T_i, i=1,\ldots,k\}$ by induction on $k$. 

If $k=1$ or $k=2$, the statement holds trivially from the assumption that the trees pairwise intersect. We treat the case $k=3$ separately as it is going to be used in the inductive step. Let $v_{12},v_{13}\in V(T)$ be vertices in $T_1\cap T_2$ and $T_1\cap T_3$ respectively. The geodesic $[v_{12},v_{13}]$ is contained in $T_1$, and by \cite[Lemma 2.8]{NewHorizons} we have that $[v_{12},v_{13}]\cap T_2\cap T_3\neq \varnothing$, hence $T_1\cap T_2\cap T_3\neq \varnothing$.
  
Suppose by induction that the result holds for each set with less than $k$ trees and consider the set $S_k=\{T_i \mid i=1,\ldots,k\}$. Define $\overline{T}=T_{k-1}\cap T_k$. Notice that by \cite[Proposition 2.4.9]{RibesBook}, the intersection of any family of pro-$\mathcal{C}$ subtrees is still a pro-$\mathcal{C}$ subtree (possibly empty) and so $\overline T$ is a pro-$\mathcal{C}$ tree. By induction (using the case $k=3$) we have that $\overline{T}\cap T_i\neq \varnothing$ for all $i\in \{1,\ldots,k-2\}$, hence we can apply the inductive hypothesis to the family $\overline{S}_k=\{\overline{T},T_1, \ldots,T_{k-2}\}$, which has by definition the same intersection as the family $S_k$ and the result follows.
\end{proof}

A pro-$\mathcal{C}$ group $G$ acts on a pro-$\mathcal{C}$ tree $\Gamma$ if it respects the incident maps, i.e. $g d_i=d_i g$ for $i\in \{0,1\}$, $g\in G$, and the action is continuous.
If an element $g\in G$ fixes at least a point of $\Gamma$ we say that $g$ is \emph{elliptic}, on the other hand, if $g$ does not fix any point, then $g$ is a hyperbolic element. We moreover say that a subgroup $H\leq G$ is elliptic if the whole subgroup fixes a point of $\Gamma$.\\
Whenever an element $g\in G$ (or a subgroup $H\leq G$) is elliptic, we can consider the set of fixed points $T^g$ (respectively $T^H$), that is a pro-$\mathcal{C}$ tree by  \cite[Theorem 4.1.5]{RibesBook}. 

When studying actions of groups on trees, we often need to restrict to minimal invariant subtrees, whose existence is guaranteed by the following lemma.

\begin{lemma}[Proposition 2.4.12 of \cite{RibesBook}] \label{lem: existence minimal subtrees}
If $G$ is a pro-$\mathcal{C}$ group acting on a pro-$\mathcal{C}$ tree $\Gamma$, then there exists a minimal $G$-invariant pro-$\mathcal{C}$ subtree $\Delta$ of $\Gamma$. If $\Delta$ contains more than one vertex, then it is unique.
\end{lemma}

The action of $G$ on a pro-$\mathcal{C}$ tree $\Gamma$ is \emph{irreducible} if $\Gamma$ has no proper $G$-invariant subtrees. From now on, for each subset $S$ of a group $G$ acting on a pro-$\mathcal{C}$ tree $T$, we denote by $T_S$ the minimal pro-$\mathcal{C}$ subtree on which $\langle S \rangle \leq G$ acts. Similarly to the abstract case, when elements commute we can obtain some additional information on their action.

\begin{lemma}[Invariant trees of commuting elements] \label{lem: culler vogtmann}
 Let $G$ be a pro-$\mathcal{C}$ group acting faithfully on a pro-$\mathcal{C}$ tree $T$.
\begin{enumerate}
  \item Let $g,h\in G$ be such that $h$ normalises $\langle g \rangle$, then $h$ leaves $T_g$ invariant and in particular, if $[g,h]=1$ then $T_g=T_h$.
  \item Let $S=\{g_1,\ldots,g_k\}$ be a set of elements such that the action of each $g_i$, $i\in \{1,\ldots,k\}$, is elliptic. If $[g_i,g_j]=1$ for each $i,j\in \{1,\ldots,k\}$, then there exists a vertex of $T$ fixed by the whole set $S$.
\end{enumerate}
\end{lemma}

\begin{proof}
Item (1) follows immediately by observing that $h\cdot (T_g)=T_{hgh^{-1}} \subseteq T_g$.\\
We first prove Item (2) for two elements $g_1,g_2\in G$. If both $g_1$ and $g_2$ are elliptic, consider the subtrees $T^{g_1}$ and $T^{g_2}$ fixed by $g_1$ and $g_2$ respectively; by Item (1) we have that $T^{g_1}$ is a non-empty pro-$\mathcal{C}$ subtree invariant under the action of $g_2$. By \cite[Corollary 4.1.9]{RibesBook}, $g_1$ fixes a vertex of $T^{g_2}$, hence $T^{g_2} \cap T^{g_1}$ is not trivial.\\
Applying the case $k=2$ to each pair, we have that $T^{g_i}\cap T^{g_j}\neq \varnothing$ and $g_i$ and $g_j$ fix pointwise the intersection for each $i,j\in \{1,\ldots,k\}$ so we can apply Lemma \ref{lem: Helly} to the set $\{T^{g_1},\ldots,T^{g_k}\}$ and conclude that $\bigcap_{i\in I} T_i \neq \varnothing$ and each $g_i$ fixes this intersection thus Item (2) follows.
\end{proof}

Let $\Delta=(V(\Delta), E(\Delta))$ be a graph. We set $m\in \Delta$ if $m\in V(\Delta)$ or $m\in E(\Delta)$. 

A finite graph of pro-$\mathcal{C}$ groups $(\mathcal{G},\Delta)$ over a finite abstract graph $\Delta$ is a collection of pro-$\mathcal{C}$ groups $\mathcal{G}(m)$ for each $m\in \Delta$, and continuous monomorphisms $\partial_i: \mathcal{G}(e)\longrightarrow \mathcal{G}(d_i(e))$ for each edge $e\in E(\Delta)$, $i\in\{0,1\}$. We only work with finite graphs of pro-$\mathcal{C}$ groups, in the sense that the graph $\Delta$ is finite, but it is possible to define an analogous concept for graphs of pro-$\mathcal{C}$ groups over profinite graphs $\Delta$ (see Chapter 6 of \cite{RibesBook}). A graph of groups is \emph{reduced} if edge groups corresponding to edges that are not loops are properly contained in adjacent vertex groups.

\begin{definition}[Pro-$\C$ fundamental group of a graph of pro-$\C$ groups]
Given a finite graph of pro-$\mathcal{C}$ groups $(\mathcal{G},\Delta)$, we define its pro-$\mathcal{C}$ fundamental group $G= \Pi_1(\mathcal{G},\Delta)$ as follows. Fix a maximal subtree $D$ of $\Delta$; then $G$ is a pro-$\mathcal{C}$ group, together with a collection of continuous homomorphisms
$$\nu_m: \mathcal{G}(m)\longrightarrow G\quad (m\in \Delta)$$
and a continuous map $E(\Delta) \longrightarrow G$, denoted $e\mapsto t_e$ ($e\in E(\Delta)$), such that $t_e=1$ if $e\in E(D)$, and such that
$$(\nu_{d_0 (e)}\partial_0)(x)= t_e(\nu_{d_1 (e)}\partial_1)(x)t_e^{-1}\quad \forall x\in \mathcal{G}(e), \ e\in E(\Delta);$$
that satisfies the following universal property:\\
  whenever we have
    \begin{itemize}
        \item a pro-$\mathcal{C}$ group $H$,
        \item a collection of continuous homomorphisms $\beta_m: \mathcal{G}(m)\longrightarrow H$, $(m\in \Delta)$,
        \item a map $e\mapsto s_e$ ($e\in E(\Delta)$) with $s_e=1$ if $e\in E(D)$, and
        \item $(\beta_{d_0 (e)}\partial_0)(x)= s_e(\beta_{d_1
        (e)}\partial_1)(x)s_e^{-1}\quad \forall x\in \mathcal{G}(e), \ e\in
        E(\Delta),$
    \end{itemize}
    then there exists a unique continuous homomorphism $\delta : G\longrightarrow  H$ with $\delta(t_e)= s_e$ $(e\in E(\Delta))$ such that for each $m\in\Delta$ the diagram
    $$\xymatrix{&
    G  \ar[dd]^\delta   \\  \mathcal{G}(m)  \ar[ru]^{\nu_m}
    \ar[rd]_{\beta_m }\\ &H }$$
    commutes.
\end{definition}

It was proven in \cite{Melnikov} that this definition does not depend on the choice of the maximal subtree $D$, moreover the existence and uniqueness of this group is proven in \cite[Proposition 6.2.1 and Theorem 6.2.4]{RibesBook}.

One can construct the fundamental group of a graph of pro-$\mathcal{C}$ groups by iterating two operations, namely pro-$\mathcal{C}$ amalgamated products and pro-$\mathcal{C}$ HNN extensions, denoted by $G_1\amalg_H G_2$ and $HNN(G_1,H,f)$ respectively, and where $G_1$ and $G_2$ are pro-$\mathcal{C}$ groups, $H\leq G_1$, and $f:H\to H'\leq G_1$ is an isomorphism. Both of these constructions are defined by means of a universal property and can be obtained as a certain pro-$\mathcal{C}$ completion of the abstract amalgamated product and HNN extension of the corresponding groups. We refer to Sections 9.2 and 9.4 of \cite{RZBook} for the precise definitions and basic properties.\\
It is important to remark that, contrary to the abstract case, the factors $G_1$ and $G_2$ (resp. the base group $G_1$) do not necessarily embed into $G_1\amalg_H G_2$ (resp. $HNN(G_1, H,f)$). Whenever they embed, the amalgamated product (resp. HNN extension) is said to be \emph{proper}. 
Some necessary and sufficient conditions for pro-$\mathcal{C}$ amalgamated products and HNN extensions to be proper were described in \cite[Theorem 9.2.4 and Proposition 9.4.3]{RZBook}. We remark that properness is assured if the amalgamated subgroup $H$ is a virtual retract of $G_1$ and $G_2$ (as $G_1$ and $G_2$ would induce the full pro-$\mathcal{C}$ topology on $H$ and the hypothesis of Thm 9.2.4 in \cite{RZBook} hold in this case).\\
Abstract Bass-Serre theory relates fundamental groups of graphs of groups with groups acting on trees. Such a relation is true for the pro-$\C$ case assuming that the action on a pro-$\C$ tree is cofinite and not true in general.
Namely given a fundamental group of a graph of pro-$\C$ groups $(\mathcal{G}, \Delta)$, there is a natural pro-$\C$ tree $T$ on which it acts. The construction of this tree, called the \emph{standard pro-$\mathcal{C}$ tree}, is described in Chapter 6 of \cite{RibesBook}.  The converse is true for the cofinite action.

If the fundamental group of the graph of pro-$\mathcal{C}$ groups is a pro-$\mathcal{C}$ amalgamated product $G=G_1\amalg_H G_2$ or a pro-$\mathcal{C}$ HNN extension $G=HNN(G_1,H,f)$, then each vertex stabiliser $G_v$ of a vertex $v$ is a conjugate of $G_1$ or $G_2$ (or of $G_1$ if $G=HNN(G_1,H,f)$) and each edge stabiliser $G_e$ is a conjugate of $H$.

Abstract Bass-Serre theory is extremely useful for studying the structure of subgroups of fundamental groups of graphs of groups. The same is true for the pro-$\mathcal{C}$ version of Bass-Serre theory, and the main tool is Theorem 7.1.7 of \cite{RibesBook}. We state the applications of these results to the case when the group acting on the pro-$\mathcal{C}$ tree is a pro-$\mathcal{C}$ amalgamated product or HNN extension. As usual, we denote by $\widehat\Z_\mathcal{C} =\prod_{p\in \pi(\mathcal{C})}\mathbb{Z}_p$ the pro-$\mathcal{C}$ completion of $\mathbb{Z}$ for any set of primes $\pi(\mathcal{C})$.

\begin{theorem}[Theorem 4.7 of \cite{NewHorizons}] \label{thm: RZ structure amalgam}
Let $K$ be a subgroup of a proper free amalgamated pro-$\mathcal{C}$ product $G=G_1\amalg_H G_2$ of pro-$\mathcal{C}$ groups. Then one of the following holds:
\begin{enumerate}
    \item $K\leq gG_ig^{-1}$ for $g\in G$ and $i\in \{1,2\}$;
    \item $K$ has a non-abelian free pro-$p$ subgroup $P$ for a certain $p\in \pi(\mathcal{C})$ such that $P\cap gG_ig^{-1}=1$ for all $g\in G$ and $i\in \{1,2\}$;
    \item there exists a subgroup $H_0 \trianglelefteq K$ (which is the kernel of the action of $K$ on $T_K$) that is contained in a conjugate of $H$ and such that $K/H_0$ is solvable and isomorphic to a projective group $\mathbb{Z}_\sigma \rtimes \mathbb{Z}_\rho$ ($\sigma, \rho \subseteq \pi(\mathcal{C})$ with $\sigma \cap \rho = \varnothing$) or $\mathbb{Z}_\sigma\rtimes C_n$ (with $\sigma \subseteq \pi(\mathcal{C})$ and $C_n$ a finite cyclic group). In the last case, it can be a profinite Frobenius group or, if $C_n=C_2$ and $2\in \sigma$, an infinite dihedral pro-$\sigma$ group.
\end{enumerate}
\end{theorem}

\begin{theorem}[Theorem 4.8 of \cite{NewHorizons}] \label{thm: RZ structure HNN}
Let $K$ be a subgroup of a proper pro-$\mathcal{C}$ HNN extension $G=HNN(G_1,H,f)$. Then one of the following holds:
\begin{enumerate}
    \item $K\leq gG_1g^{-1}$ for $g\in G$;
    \item $K$ has a non-abelian free pro-$p$ subgroup $P$ for $p\in \pi(\mathcal{C})$ such that $P\cap gG_1g^{-1}=1$ for all $g\in G$;
    \item there exists a subgroup $H_0 \trianglelefteq K$ (which is the kernel of the action of $K$ on $T_K$) that is contained in a conjugate of $H$ and such that $K/H_0$ is solvable and isomorphic to a projective group $\mathbb{Z}_\sigma \rtimes \mathbb{Z}_\rho$ ($\sigma, \rho \subseteq \pi(\mathcal{C})$ with $\sigma \cap \rho = \varnothing$) or $\mathbb{Z}_\sigma\rtimes C_n$ (with $\sigma \subseteq \pi(\mathcal{C})$ and $C_n$ a finite cyclic group). In the last case, it can be a profinite Frobenius group or, if $C_n=C_2$ and $2\in \sigma$, an infinite dihedral pro-$\sigma$ group.
\end{enumerate}
\end{theorem}
A useful remark is that, in the third case of the previous theorems, $H/H_0$ is torsion-free if and only if it is isomorphic to $\mathbb{Z}_\sigma \rtimes \mathbb{Z}_\rho$. In this case, as this is a projective group, we have that $H\cong H_0 \rtimes (\mathbb{Z}_\sigma \rtimes \mathbb{Z}_\rho)$.

Finally, we record the following observation.

\begin{lemma}\label{lem: commutation of a path}
Let $G=G_1\amalg_H G_2$ be a proper amalgamated pro-$\mathcal{C}$ product of two pro-$\mathcal{C}$ groups $G_1$ and $G_2$ and let $T$ be the standard pro-$\mathcal{C}$ tree associated with this splitting. Let $g_1,\ldots,g_k$ be a sequence of elliptic elements such that $[g_i,g_{i+1}]=1$ for all $i\in\{1,\ldots, k-1\}$. Then there are some vertices $v_1,\ldots, v_k\in V(T)$ (not necessarily distinct) such that $g_1\in G_{v_1}$ and $g_i\in G_{t_i}$ for each $t_i\in [v_{i-1},v_i]$.
\end{lemma}
\begin{proof}
By Lemma \ref{lem: culler vogtmann} there exists a vertex $v_i$ stabilized by each pair of commuting elements $g_{i},g_{i+1}$ for each $i\in \{1,\ldots, k-1\}$. Define $v_k$ to be any vertex stabilized by $g_k$. In this setting, $g_i$ stabilizes both $v_{i-1}$ and $v_i$, hence it stabilizes the whole subtree $[v_{i-1},v_i]$ by \cite[Corollary 4.1.6]{RibesBook}.
\end{proof}

\section{Basic on pro-\texorpdfstring{$\C$}{C} RAAGs}

The aim of this section is to describe basic properties pro-$\C$ RAAGs. The abstract version of the definitions and results that we discuss can be found, for example, in \cite{Charney}.

Let $\Gamma=(V(\Gamma), E(\Gamma))$ be a finite simplicial graph where $V(\Gamma)$ and $E(\Gamma)$ are the set of vertices and edges respectively. A subgraph $\Delta < \Gamma$ is called $\emph{full}$ if for all $e\in \Gamma$ with $d_0(e),d_1(e)\in \Delta$ we have that $e\in \Delta$. Notice that full subgraphs are uniquely determined by the subset of vertices $V(\Delta)$ of $V(\Gamma)$.

\begin{definition}[Right-angled Artin pro-$\mathcal{C}$ groups]\label{defn: pro-C RAAG}
The \emph{right-angled Artin pro-$\mathcal{C}$ group} (pro-$\mathcal{C}$ RAAG for short) $G_\Gamma$ is the pro-$\mathcal{C}$ group given by the pro-$\mathcal{C}$ presentation
$$G_\Gamma=\langle V(\Gamma)|[u,v]=1 \mbox{ if and only if $u$ and $v$ are adjacent in $\Gamma$}\rangle.$$

\end{definition}

We recall some standard terminology.

\begin{definition}[Canonical Generators]
The generators associated with the vertices of $\Gamma$ are called \emph{canonical generators} and, abusing the notation, we denote them with the same letter as the corresponding vertex.
\end{definition}

\begin{definition}[Standard subgroups]\label{def: standard subgroup}
A subgroup of $G_\Gamma$ is called a \emph{standard subgroup} if it is the subgroup generated by a subset $V'\subseteq V(\Gamma)$. If $\Gamma = \varnothing$, by convention we set $G_\Gamma$ to be the trivial subgroup.
\end{definition}
Abusing the notation, if $S\subseteq V(\Gamma)$, we denote by $G_{S}$ the standard subgroup generated by the full subgraph generated by $S$.

\begin{lemma}[Properties of standard subgroups]\label{lem: Raags as completions}
Let $G_\Gamma$ be a pro-$\mathcal{C}$ RAAG. Then:
\begin{enumerate}
    \item $G_\Gamma$ is the pro-$\mathcal{C}$ completion of the abstract RAAG $G(\Gamma)$;
    \item the standard subgroup generated by a subset of vertices $V'\subseteq V(\Gamma)$ is the pro-$\mathcal{C}$ RAAG $G_\Delta$ generated by the full subgraph $\Delta\subseteq \Gamma$ determined by $V'$;
    \item the standard subgroups of $G_\Gamma$ are retracts;
    \item the intersection of standard subgroups is a standard subgroup (possibly trivial).
\end{enumerate}
\end{lemma}

\begin{proof}
Given a group $G$, we denote by $\widehat{G}$ its pro-$\mathcal{C}$ completion.
\begin{enumerate}
    \item Follows from the pro-$\mathcal{C}$ presentation (see Definition \ref{defn: pro-C RAAG}).

\item In the abstract case, the subgroup of $G(\Gamma)$ generated by $V'$ is exactly $G(\Delta)$, see for example Corollary 2.11 of \cite{Koberda}. As this subgroup is a retract of $G(\Gamma)$, the pro-$\mathcal{C}$ topology of $G(\Gamma)$ induces on it the full pro-$\mathcal{C}$ topology, so the pro-$\mathcal{C}$ subgroup $\langle V'\rangle \leq G_\Gamma$ is $\widehat{G(\Delta)}$, that by (1) coincides with $ G_\Delta$.
\item The map $\mbox{pr}_{\Delta}:G_\Gamma \to G_\Delta$ whose restriction to $G_\Delta$ is the identity and such that $\mbox{pr}_{\Delta}(v)=1$ for each $v\in V(\Gamma)\smallsetminus V'$ is surjective. Since by (2) $G_\Delta$ is a subgroup of $G_\Gamma$, we have that $\mbox{pr}_\Delta$ is a retraction onto $G_\Delta$.
\item Consider two standard subgroups $G_\Delta, G_\Lambda$ of $G_\Gamma$. By (3), a non-trivial element $g$ of $G_\Gamma$ is in $G_\Delta\cap G_\Lambda$ if and only if $\mbox{pr}_{\Delta}(\mbox{pr}_{\Lambda}(g))=g$, but this composition of maps corresponds exactly to $\mbox{pr}_{\Delta\cap \Lambda}(g)$, and therefore $G_\Delta\cap G_\Lambda= G_{\Delta\cap \Lambda}$.
\end{enumerate}   
\end{proof}

It follows from the pro-$\C$ version of Theorem 9.2.4 in \cite{RZBook} that, if $H$ is a retract of two groups $G_1$ and $G_2$, then a pro-$\C$ $G_1\amalg_H G_2$ is a proper pro-$\mathcal{C}$ amalgamated product. Similarly, it follows from Theorem 9.4.3 that pro-$\C$ HNN-extension $HNN(G_1,H,f)$ is proper if $H$ is a retract of $G_1$. As standard subgroups of RAAGs are retracts, we deduce the following.

\begin{corollary}\label{cor: proper_retraction}
Let $G_\Gamma$ be a pro-$\mathcal{C}$ RAAG. If $G_\Gamma$ is a pro-$\mathcal{C}$ amalgamated product $G_1\amalg_H G_2$ or a pro-$\mathcal{C}$ HNN extension $HNN(G_1,H,f)$ with $G_1,G_2,H,f(H)$ standard subgroups of $G_\Gamma$, then the free product with amalgamation or HHN extension is proper.
\end{corollary}

\begin{lemma}[Support is well-defined]\label{lem: minimal standard subgroups}
    Let $G_\Gamma$ be a pro-$\mathcal{C}$ RAAG and let $g\in G_\Gamma$. Then there exists a unique minimal standard subgroup containing $g$. Moreover, there exists an element $h$ in the conjugacy class of $g$ whose corresponding minimal standard subgroup is contained in each standard subgroup containing conjugates of $g$.
\end{lemma}

\begin{proof}
    The unique minimal standard subgroup containing $g$ is the intersection of all the standard subgroups containing it, and this intersection is still a standard subgroup by Lemma \ref{lem: Raags as completions}.
    Suppose now that $\Delta_1, \Delta_2$ are full subgroups of $\Gamma$ such that $g\in G_{\Delta_1}$ and $g^t\in G_{\Delta_2}$ for $t\in G_\Gamma$. We claim that there exists $s\in G_\Gamma$ such that $g^s\in G_{\Delta_1 \cap \Delta_2}$. Indeed let $\mbox{pr}_{\Delta_1}$ be the retraction of $G_\Gamma$ to $G_{\Delta_1}$ and define $s=\mbox{pr}_{\Delta_1}(t)$. Then $g^s=\mbox{pr}_{\Delta_1}(g^t)\in G_{\Delta_1 \cap \Delta_2}$. In order to prove the lemma it suffices to apply this observation to the lattice of full subgraphs of $\Gamma$ containing a conjugate of $g$. Notice that if $g=1$ we have that $g\in G_\varnothing$ and by convention, the standard subgroup generated by the empty set is the trivial group.
\end{proof}

\begin{definition}[Support of an element]
Let $g$ be an element of a (pro-$\mathcal{C}$) RAAG $G_{\Gamma}$.
The \emph{support} $\alpha(g)$ of $g$ is the set of canonical generators of the unique minimal standard subgroup of $G_{\Gamma}$ containing $g$.

In view of Lemma \ref{lem: minimal standard subgroups}, in any conjugacy class there exists an element $g$ such that $\alpha(g)\subseteq \alpha(g^t)$ for each $t\in G$, in this case we say that $g^t$ is an \emph{element of minimal support among its conjugates}.
\end{definition}

\begin{definition}[Links and stars]
Let $g$ be an element of a (pro-$\mathcal{C}$) RAAG $G_{\Gamma}$.
The link $\link(g)$ of $g$ is the set of vertices of $\Gamma \smallsetminus \alpha(g)$ that are adjacent to each of the vertices in $\alpha(g)$.\\
If $v$ is a canonical generator, we denote by $\mbox{Star}(v)$ the full subgraph generated by $\link(v)\cup v$.
\end{definition}

\begin{remark} \label{rem: hyperbolic splitting}
If $v\in V(\Gamma)$, we can split $G_\Gamma$ as a pro-$\mathcal{C}$ HNN extension as
\begin{equation} \label{eq: splitting with a vertex}
    G_\Gamma=HNN(G_{\Gamma \smallsetminus \{v\}}, G_{\tiny{\link(v)}}, id)
\end{equation}
with stable letter $v$
and by Corollary \ref{cor: proper_retraction} this is a proper pro-$\mathcal{C}$ HNN extension.
It follows that if $g$ is an element with minimal support among its conjugates and $v\in \alpha(g)$, then Theorem \ref{thm: RZ structure HNN} guarantees that its action on the standard pro-$\mathcal{C}$ tree $T$ associated with this splitting is hyperbolic. 

\medskip
We can also split $G_\Gamma$ as a free pro-$\C$ product with amalgamation
\begin{equation} \label{eq: splitting amalg with a vertex}
    G_\Gamma=G_{\Gamma \smallsetminus \{v\}}\coprod_ {G_{\tiny{\link(v)}}} (G_{\tiny{\link(v)}}\times \langle v\rangle).
\end{equation} 
\end{remark}

We shall need the following

\begin{lemma}\label{normal closure} Let $G=G_{\Gamma}$ be a (pro-$\mathcal{C}$) RAAG and $v$ a vertex of $\Gamma$.  Let $y$ be an element of $G_{\tiny{\link(v)}}$. Then $\langle y\rangle\llangle v\rrangle=(F_0\times \langle y\rangle)\amalg F $, where $\llangle v\rrangle$ is the normal closure of $v$ in $G$  and $F_0, F$ are  free pro-$\C$. 

\end{lemma}

\begin{proof} By the pro-$\C$ version of the subgroup theorem for normal subgroups applied to \eqref{eq: splitting amalg with a vertex} (see \cite[Theorem B]{Zal95}), using that $\llangle v \rrangle$ does not intersect any conjugate of $G_{\tiny{\mbox{Link}(v)}}$, we obtain that $\llangle v \rrangle$ is a free pro-$\C$ group $F$, with basis $\{v^{G_{\Gamma \smallsetminus \{v\}}}\}=\{v^{(G_{\tiny{\mbox{Link}(v)}}\backslash G_{\Gamma \smallsetminus \{v\}})}\}$ as $v$ is centralized by $G_{\tiny{\mbox{Link}(v)}}$. Note that $G_{\Gamma \smallsetminus \{v\}}$ acts on $G_{\tiny{\mbox{Link}(v)}}\backslash G_{\Gamma \smallsetminus \{v\}}$ by multiplication and hence does $\langle y\rangle$. Let $S\subseteq G_{\tiny{\mbox{Link}(v)}}\backslash G_{\Gamma \smallsetminus \{v\}}$ be the subset of fixed points for $y$ and put  $(R,*)=(G_{\tiny{\mbox{Link}(v)}}\backslash G_{\Gamma \smallsetminus \{v\}})/ S$ to be the pointed profinite space with   $S$ collapsed  as a distinguished point.  We can write  $\langle y\rangle\llangle v\rrangle$ as $\langle y\rangle\llangle v\rrangle=(F(S)\amalg F(R, *))\rtimes \langle y\rangle$. Let $(X,*)=(R,*)/\langle y\rangle$ and $F(X,*)$ a free pro-$\C$ group over the pointed profinite space $(X,*)$. Since $\langle y\rangle$ acts freely on the pointed profinite space $(R,*)$, the free pro-$\C$ product $F(X,*)\amalg \langle y \rangle=F(R, *)\rtimes \langle y\rangle$ as the normal closure of $F(X,*)$ in $F(X,*)\amalg \langle y \rangle$  is exactly  $  F(R, *)$ by \cite[Theorem B]{Zal95}. Thus $\langle y\rangle\llangle v\rrangle=(F(S)\amalg F(R, *))\rtimes \langle y\rangle=(F(S)\times  \langle y\rangle) \amalg F(X, *)$ as required.
\end{proof}

Abstract right-angled Artin groups are torsion-free, but the pro-$\C$ completion of torsion-free groups is not always torsion-free (even the profinite completion as shown in \cite{Lubotzky},\cite{ChatzidakisTorsion}). However, in the case of pro-$\mathcal{C}$ RAAGs this is true.

\begin{theorem}\label{prop: Raags torsionfree}
Pro-$\mathcal{C}$ RAAGs are torsion-free profinite groups.
\end{theorem}

\begin{proof}
A pro-$\mathcal{C}$ RAAG is the pro-$\mathcal{C}$ completion of the corresponding (abstract) RAAG. In \cite{DuchampKrob}, the authors proved that abstract RAAGs are residually-(finitely generated torsion-free nilpotent), and hence the pro-$\C$ completion of a RAAG embeds in a direct product of the pro-$\mathcal{C}$ completions of finitely generated torsion-free nilpotent groups. By \cite[Theorem 4.7.10]{RZBook}, the profinite completion $\widehat N$ of a finitely generated torsion-free nilpotent group $N$ is torsion-free. But $\widehat N=\prod_p \widehat N_p$ is the direct product of the pro-$p$ completions and the pro-$\C$ completion of $N$ is the direct product $\prod_{p\in \pi(\C)} N_p$. Hence the pro-$\C$ completion of $N$ is torsion-free.
\end{proof}

We recall that a subgroup $H\leq G$ is \emph{isolated} (or isolated in $G$) if whenever $g^k\in H$ for a certain $g\in G$, then $g\in H$.

\begin{lemma}\label{lem: standard subgroups isolated}
Standard subgroups of pro-$\C$ RAAGs are isolated.
\end{lemma}

\begin{proof}
     Let $G=G_\Gamma$. The theorem is equivalent to the statement that for every $g\in G$ and $k\in \widehat{\mathbb{Z}}_\C$, $\alpha(g)=\alpha(g^k)$. Suppose this is not true, so that there exists a vertex $v\in \alpha(g)\smallsetminus \alpha (g^k)$. Consider the standard pro-$\C$ tree $T$ associated with the splitting (\ref{eq: splitting with a vertex}) and notice that $g^k\in G_{\Gamma\smallsetminus \{v\}}$ stabilizes the vertex $t$ of $T$, which is stabilized by the standard subgroup $G_{\Gamma \smallsetminus \{v\}}$. Note that  $g$ can not act  hyperbolically on $T$ since $g^k$ is elliptic (cf. \cite[Corollary 2.5 (i)]{KZ23}). So
 both $g$ and $g^k$ act elliptically. In this case, we can argue by induction on the number of generators of $G_\Gamma$ and suppose that standard subgroups of pro-$\C$ RAAGs with at most $|V(\Gamma)|-1$ generators are isolated. By induction, and using that isolation is invariant by conjugation,  edge stabilizers are isolated in adjacent vertex stabilizers. Denote by $f$ the projection onto the standard subgroup $G_{\Gamma\smallsetminus \{v\}}$ (modulo the normal closure of $v$), and set $x=f(g)$.

    Notice that the set of edges of a standard pro-$\C$ tree associated with a splitting is compact by construction, and therefore for any vertex of any of its pro-$\C$ subtrees containing at least an edge, there always exists an edge of the subtree adjacent to it (see \cite[Proposition 2.1.6 (c)]{RibesBook}). As  $x^k=g^k$ fixes both the vertex fixed by $g$ and the vertex $t$ stabilized by $G_{\Gamma\smallsetminus \{v\}}$, there exists a stabilizer of an edge $e$ adjacent to $t$ that contains $x^k$. By inductive hypothesis on isolation  $x$ fixes $e$   and    $he$ is fixed by  $G_{\tiny{\mbox{Link}(v)}}$ for some $h\in G_{\Gamma\smallsetminus \{v\}}$. Then   conjugating $g$  and $x$ by $h$ if necessary we may assume that $x\in G_{\tiny{\mbox{Link}(v)}}$.  

     Denoting by $\llangle v \rrangle$ the normal closure of $v$ in $G_\Gamma$, we have  $g\in \langle x \rangle \llangle v \rrangle$. By Lemma \ref{normal closure}
    $$
    \langle x \rangle \llangle v \rrangle =  \big(\langle x \rangle \times F_0 \big) \coprod F,$$
    where $F,F_0$ are free pro-$\C$.  
    Now $x$ and $g$ centralize $g^k=x^k$ and $x^k$ clearly lies in the factor $\langle x \rangle \times F_0$ of the free product. By \cite[Theorem B]{HeZa85}, the element $g$ must also lie in $\langle x \rangle \times F_0$.
    Let $c\in F_0$ be such that $g=x^{\ell}\cdot c$. Then $x^k=g^k=x^{k \ell}\cdot c^k$. Since by Proposition \ref{prop: Raags torsionfree} $G$ is torsion-free, it follows that $c=1$ and $\ell=1$. Therefore, $g=x$ , contradicting the choice of $g$. This proves the statement.
    
\end{proof}

\begin{theorem}[Strong Tits' alternative]\label{thm: solvable subgroups are metabelian}
Let $H$ be a closed subgroup of a pro-$\mathcal{C}$ RAAG $G_\Gamma$ that does not contain a free non-abelian pro-p subgroup for
any prime $p$. Then $H$ is metabelian and polycyclic. Moreover, if $H$ is pro-$p$, then $H$ is abelian.
\end{theorem}

\begin{proof} We use induction on the number of vertices of $\Gamma$. If $\Gamma$ consists of a single vertex, then any subgroup of $G_\Gamma$ is pro-$\C$ cyclic and the result follows. As we observed in Remark \ref{rem: hyperbolic splitting}, $G_\Gamma=HNN(G_{\Gamma\setminus \{v\}}, G_{\tiny{\mbox{Link}}(v)},id)$ for an arbitrarily chosen vertex $v$. If $H$ is conjugate to a subgroup of $G_{\Gamma\setminus \{v\}}$ then we deduce the result from the induction hypothesis. Otherwise, by Theorem \ref{thm: RZ structure HNN} there exists a normal subgroup $H_0\trianglelefteq H$ contained in some conjugate of $\mbox{Link}(v)$ such that $H/H_0$ is metacyclic.  Lemma \ref{lem: standard subgroups isolated} guarantees that $G_{\tiny{\mbox{Link}}(v)}$ is an isolated subgroup of $G_\Gamma$ and so $H_0$ is isolated in $H$. It follows that the only possibility for $H/H_0$ in Theorem \ref{thm: RZ structure HNN} is the projective group $\Z_\sigma\rtimes \Z_\rho$, so that $H=H_0\rtimes P$, with $P\cong \Z_\sigma\rtimes \Z_\rho$ and we may assume without loss of generality that $[P,P]\cong \Z_\sigma$.

Consider the projection $f:G_\Gamma \longrightarrow G_{\Gamma\setminus \{v\}}$ and let $K$ be the kernel of $f$. The image $f([H,H])$ is abelian and finitely generated by the induction hypothesis. Since $K\cap H$ and $H_0$ are normal in $H$ and do not intersect, as $H_0\subseteq G_{\tiny{\mbox{Link}}(v)}$, we have that $(K\cap H)H_0=(K\cap H)\times H_0$. Then the commutator subgroup $[H,H]= [H_0,H_0][H_0,P][P,P]$, and $[H_0,H_0] [H_0,P]$ is abelian and polycyclic because it is in the image of $f([H,H])$. Thus we just need to show that $[P,P]$ centralizes $[H_0,H_0] [H_0,P]$. To see this observe that $f([P,P])$ is torsion-free and so $[P,P]\cong (K\cap [P,P])\times f([P,P])$. Now $K$ centralizes $H_0$ and $f([P,P])$ centralizes $f([H_0,H_0] [H_0,P])=[H_0,H_0] [H_0,P]$ by inductive hypothesis, so we deduce that $[H,H]$ is abelian.

Suppose now $H$ is pro-$p$. Then $P\cong \Z_p$. But the action of $P$ on $H_0$ is the same as the action of $f(P)$ on $H_0$ which is trivial since $f(H)$ is abelian by the inductive hypothesis. Therefore $H=H_0\times P$ is abelian. 
\end{proof}

The next result describes two-generated   pro-$p$ subgroups of pro-$\C$ RAAGs. The pro-$\C$ case is more complicated, as there are two generated metabelian pro-$\C$ subgroups but also pro-$\C$ subgroups of the form $(\Z_{\sigma_1} \coprod \Z_{\tau_1}) \times \cdots \times (\Z_{\sigma_\ell} \coprod \Z_{\tau_\ell})$ for $\sigma_i, \tau_i \subseteq \pi(\C)$, with $\sigma_i$ pairwise disjoint and $\tau_i$ pairwise disjoint.

\begin{theorem}[2-generated pro-$p$ subgroups]\label{thm: two generated subgroups}
Let $H$ be a two-generated  pro-$p$ subgroup of a pro-$\C$ RAAG $G_\Gamma$. Then $H$ is either free pro-$p$ or free abelian.
\end{theorem}

\begin{proof} We use induction on the number of vertices of $\Gamma$. If $\Gamma$ is a vertex, then $G_\Gamma$ is abelian and so is each subgroup so the statement holds. As we noticed in Remark \ref{rem: hyperbolic splitting}, chosen an arbitrary vertex $v$, we can obtain a decomposition $G_\Gamma=HNN(G_{\Gamma\setminus \{v\}}, G_{\tiny{\mbox{Link}}(v)},id)$. If $H$ is conjugate to a subgroup of $G_{\Gamma\setminus \{v\}}$ then we deduce the result from the induction hypothesis. Otherwise, $H$ acts non-trivially on the standard pro-$\C$ tree associated with HNN extension $G_\Gamma=HNN(G_{\Gamma\setminus \{v\}}, G_{\tiny{\mbox{Link}}(v)},id)$. Considering the projection $f:G_\Gamma \longrightarrow G_{\Gamma\setminus \{v\}}$, by induction hypothesis we can deduce that $f(H)$ is either a free or abelian pro-$p$ group. If $f(H)$ is free, then we are done because by the Hopfian property (see \cite[Proposition 2.5.2]{RZBook}) the projection must be an isomorphism.

Suppose suppose now that $f(H)$ is abelian. By the induction hypothesis, we can assume that every stabiliser of a vertex in $H$ is abelian of rank at most two. By Lemmas \ref{lem: Raags as completions} and \ref{lem: standard subgroups isolated}, $G_{\tiny{\mbox{Link}}(v)}$ is an isolated subgroup of $G_\Gamma$ and so every edge stabiliser of $H$ is isolated in the vertex stabiliser of its incident vertex. As the only isolated proper subgroups of an abelian pro-$p$ group of rank two are procyclic, it follows that such are the edge stabilisers. Then by \cite[Theorem 6.8]{ChatzidakisZalesskii}, $H$ is the fundamental group of a finite graph of pro-$p$ groups $(\mathcal{H}, \Delta)$ whose vertex and edge groups are isomorphic to vertex and edge stabilisers in $H$ respectively. Assuming without loss of generality that this graph of groups is reduced, we deduce that isolation of edge groups in the incident vertex groups implies that either the edge groups have strictly smaller rank than the incident vertex groups, or they coincide (in the case when $\Delta$ contains loops).

As $H$ is two generated and edge groups are isolated, $\Delta$ cannot have more than two vertices. For the same reason, if $|V(\Delta)|=2$, as the graph of groups is reduced, the edge group between the two vertices can only be trivial. Another remark is that the stable letter of each HNN extension corresponding to a loop must be one of the two generators of the group.

If at least an edge group is trivial, then $H$ splits as a free pro-$p$ product (see \cite[Proposition 2.16]{CastZal22}) and so, being  2-generated, it has to be a free pro-$p$ product of torsion-free cyclic groups, hence it is free pro-$p$. We only have to analyse the case when no edge group is trivial.

If $\Delta$ has a single vertex, we have to analyse the case when there are zero, one, or two loops. If there are two loops, as each stable letter of an HNN extension must be one of the two generators, the vertex group must be trivial and $H$ is a free pro-$p$ group. In all of the other cases, the vertex group has either rank one or two.

Let us now consider the case when $\Delta$ has a vertex $w$ with a single loop and the vertex group has either rank one or 2.

Suppose the vertex group has rank one. Since we can assume the edge group not to be trivial, it is also abelian of rank 1 and since edge groups are retractions, it follows that the edge group coincides with the vertex group. Since pro-$\C$ RAAGs are torsion-free, it follows that $H$ splits as an HNN extension $HNN(\mathcal{H}(w), \mathcal{H}(w), id)$, which is a free abelian pro-$p$ group of rank two.

The last case to consider is when $\mathcal{H}(w)$ is free abelian of rank 2. As the group is two-generated and one generator must be the stable letter $t$ of the HNN extension, the only possibility is that $\mathcal{H}(w)=\langle x, x^t \rangle$ for some $x,t\in G$, $t$ power of the stable letter. Consider the retraction $f:G_\Gamma \to G_{\Gamma\setminus \{v\}}$. Since $\langle x, x^t \rangle$ is a free abelian group of rank 2, we have that $x=f(x), x^{f(t)}$ also generate an abelian group of rank 2 and so, in particular, $f(t)\in G_{\Gamma\setminus \{v\}}$ is nontrivial. Now, by induction hypothesis, the $2$-generated subgroup $\langle x, f(t)\rangle < G_{\Gamma\setminus\{v\}}$ is either free or free abelian. The latter case implies that $x=x^{f(t)}$ contradicting the fact that $\langle x, x^{f(t)}\rangle$ is of rank 2. If $x$ and $f(t)$ generate a free group, then $[x, x^{f(t)}] \ne 1$ contradicting that the fact that $\langle x, x^{f(t)} \rangle$ is abelian. This proves that this case cannot occur.

Since all the alternatives have been considered, the result follows.
\end{proof}

\section{Graph characterisation of the direct product decomposition of pro-\texorpdfstring{$\C$}{C} RAAGs}

Our goal is to show that the direct product decomposition of a pro-$\C$ RAAG is encoded by the defining graph. More precisely $G_\Gamma \simeq A_1 \times A_2$ where $A_1$ and $A_2$ are non-trivial pro-$\C$ groups if and only if $\Gamma$ is a join, see Theorem \ref{thm: directly decomposable}.

\begin{lemma}[Standard subgroup of a centraliser]\label{lem: centralisers}
Let $G_\Gamma$ be a pro-$\mathcal{C}$ RAAG and let $g\in G_\Gamma$ be an element with minimal support among its conjugates. Then, the centraliser of $g$ is contained in the standard subgroup generated by $\link(g)\cup \alpha(g)$. In particular, if $g=v$ is a standard generator, then $C_G(v) = G_{\tiny{\mbox{Star}(v)}}=\langle v\rangle \times G_{\tiny{\link(v)}}$.
\end{lemma}

\begin{proof}
Suppose towards contradiction that there is an element $h$ commuting with $g$ whose support is not contained in $\link(g)\cup \alpha(g)$. Then there exists $v\in \alpha(h)$ such that $v\notin \link(g)\cup \alpha(g)$. Denoting by $G_0=G_{\Gamma \smallsetminus \{v\}}$ and by $A=G_{\tiny{\link(v)}}$ and using Remark \ref{rem: hyperbolic splitting}, we have that the group $G_\Gamma$ splits as a proper HNN extension of the form
$$G_\Gamma=HNN(G_0, A, id)$$
where the action by conjugation of $v$ on $A$ is trivial. Notice that from the assumption on $h$, we have that $h\notin G_0$. We next study the action of $g$ and $h$ on the standard pro-$\mathcal{C}$ tree $T$ associated with this splitting. 

Notice that $g \in G_0$ and so $g$ is elliptic. However, $g$ cannot belong to any edge stabiliser. Indeed, otherwise, there would exist an element $t\in G_\Gamma$ such that $g^t\in A$ and in this case, since $g$ has by assumption minimal support, it would follow from Lemma \ref{lem: minimal standard subgroups} that $\alpha(g)\subseteq \alpha(g^t)\subseteq \link(v)$ and so $v\in Link(g)$ contradicting the choice of $v$. Since $g$ cannot be in any edge stabiliser, we conclude that $g$ only fixes the vertex $v$ stabilised by $G_0$, i.e. $T_g=\{v\}$. From Lemma \ref{lem: culler vogtmann} (1), $h$ has to leave $T_g=\{v\}$ invariant and, in particular, $h$ fixes $v$. Then $h$ belongs to $G_0$, a contradiction.
\end{proof}

\begin{lemma}[Direct factors are in stars]\label{direct product} Suppose a pro-$\C$ RAAG $G=G_\Gamma$  decomposes as a direct product $G_\Gamma = A_1 \times A_2$ of non-trivial groups. Then for each canonical generator $v\in \Gamma$, at least one factor $A_i$ is contained in $G_{\tiny{\mbox{Star}(v)}}$.
\end{lemma}

\begin{proof} Let $v$ be a canonical generator. Since $\alpha(v)=\{v\}$, by Lemma \ref{lem: centralisers} we have that $C_G(v) = G_{\tiny{\mbox{Star}(v)}}=\langle v\rangle \times G_{\tiny{\link(v)}}$.

Suppose that $v= a_1 \cdot a_2$ where $a_i\in A_i$, $i=1,2$. Since $C_G(v)= C_{A_1}(a_1) \times C_{A_2}(a_2)$ and $a_i\in C_{A_i}(v)$, from the description of the centraliser $C_G(v)$, we deduce that $a_i = v^{e_i} a_i'$ for $e_i\in \mathbb Z_{\pi(\C)}$ and $a_i' \in G_{\tiny{\link(v)}}$. Since $v = a_1 \cdot a_2$, we have that $e_i \ne 0$ for either $i=1$ or $i=2$; without loss of generality assume $e_1\neq 0$. Let $t$ be an element such that $t^{-1}a_1t$ has minimal support among its conjugates, we can assume $t\in A_1$ because $A_2\subseteq C_G(a_1)$. Applying Lemma \ref{lem: centralisers} we have
$$
tA_2t^{-1}=A_2\subseteq C_G(a_1)=tC_G(t^{-1}a_1t)t^{-1} \subseteq t\big(G_{\alpha(t^{-1}a_1t)}\times G_{\tiny{\link(t^{-1}a_1t)}}\big)t^{-1}.
$$
Notice that by Lemma \ref{lem: minimal standard subgroups} $\alpha(t^{-1}a_1t)\subseteq \alpha(a_1)\subseteq \mbox{Star}(v)$ and, since $v\in \alpha(t^{-1}a_1t)$, the definition of link implies that $\link(t^{-1}a_1t)\subseteq \mbox{Star}(v)$. Overall, we conclude that $A_2\subseteq G_{\tiny{\mbox{Star}(v)}}$.

\end{proof}

\begin{theorem}[Direct product decomposition]\label{thm: directly decomposable}
Let $G_\Gamma$ be a pro-$\mathcal{C}$ RAAG. Then $G_\Gamma$ has a non-trivial direct product decomposition if and only if $\Gamma$ is a join (i.e. there is $\varnothing\ne \Delta\lneq \Gamma$ such that for each $v\in \Delta$ and each $w\in \Gamma \smallsetminus \Delta$, $v,w$ are adjacent). In particular, each factor in a direct product decomposition of $G_\Gamma$ is a standard subgroup.
\end{theorem} 

\begin{proof}
The analogous result for abstract RAAGs is classical (see for example \cite[Corollary 2.15]{Koberda}). From the abstract result and Lemma \ref{lem: Raags as completions}, it is straightforward that whenever $\Gamma$ is a join, then $G_\Gamma$ splits as a direct product.

We now want to prove the converse implication. By Lemma \ref{direct product} for each canonical generator $v$, at least one among $A_1$ or $A_2$ is contained in $G_{\tiny{\mbox{Star}(v)}}$.

Let $\Gamma_{1} \subseteq V(\Gamma)$ be the set of canonical generators $v$ such that $A_1 < \mbox{Star}(v)$ and $\Gamma_2=\Gamma\smallsetminus \Gamma_1$. Then, for each canonical generator $v\in \Gamma_2$, since by definition of $\Gamma_2$ we have that $A_1 \not< \mbox{Star}(v)$, by Lemma \ref{direct product} again we conclude that $A_2 \leq G_{\tiny{\mbox{Star}(v)}}$.

For $i=1,2$ define $\Delta_i\subseteq \Gamma$ such that $G_{\Delta_i} = \bigcap_{v\in \Gamma_i} G_{\tiny{\mbox{Star}}(v)}$; by Lemma \ref{lem: Raags as completions} $G_{\Delta_i}$ is a standard subgroup and by definition it contains $A_i$ and each $v\in \Gamma_i$ is connected to each $w\in \Delta_i$. In particular, $\Delta_i$ are non-empty graphs. Notice that if there is a canonical generator $w \in \Delta_1 \cap \Delta_2$, then $w$ is by definition in the star of each vertex in $\Gamma_i$ and so $\Gamma_1, \Gamma_2 < \mbox{Star}(w)$. Hence such a canonical generator $w\in \Delta_1 \cap \Delta_2$ would be central and $\Gamma$ would decompose as a join. For this reason we can assume that $G_{\Delta_1}$ and $G_{\Delta_2}$ are disjoint and since $A_1$ and $A_2$ generate $G$, so do $G_{\Delta_1}$ and $G_{\Delta_2}$.

Hence, we can decompose $V(\Gamma)$ as the disjoint union of the (possibly empty) sets $\Gamma_2\cap \Delta_1$, $\Gamma_1 \cap \Delta_2$ and $\Lambda = (\Gamma_1 \cap \Delta_1) \cup (\Gamma_2 \cap \Delta_2)$.

Since $\Delta_i$ is non-empty for $i=1,2$, then either $\Lambda \ne \emptyset$ or $\Gamma_2\cap \Delta_1$ and $ \Gamma_1 \cap \Delta_2$ are non-empty. If at least two of the sets are non-empty, then they define a join, because each vertex in a set is connected to each vertex in the other set, because each element in $\Gamma_i$ is connected to each element in $\Delta_i$ for $i=1,2$.  

We are left to consider the case when only $\Lambda$ is non-empty so that $\Lambda=V(\Gamma)$. In this case, each vertex in $\Gamma_i$ is also in $\Delta_i$, and in particular, they are connected to each other. It follows that $\Gamma_i \cap \Delta_i = \Gamma_i = \Delta_i$ is a complete graph for $i=1,2$. Since $A_i \leq G_{\Delta_i}$ and $G_{\Delta_i}$ is abelian, so is $A_i$. Hence $G={A_1}\times A_2$ is abelian and $\Gamma$ is a complete graph and a join.
\end{proof}

These results are in line with other properties of pro-$\C$ RAAGs that can be recognized from the abstract graph. For example, abstract RAAGs split as a free product if and only if the underlying graph is disconnected, and Wilkes and Kropholler proved that the same is true for profinite RAAGs in \cite{Wilkes}. Similarly, both abstract and pro-$p$ RAAGs are coherent if and only if the underlying graph is chordal, see \cite{SnopceZalesskii}.

\section{Centralisers of an element}

In this section, we describe explicitly the structure of the centralizer of an element in a pro-$\C$ RAAG. In a free pro-$p$ group, centralisers of elements are cyclic. However, in the pro-$\mathcal{C}$ case, the situation is substantially different as the centraliser of an element does not need to be cyclic. Indeed, for example, the projective group $\Z_3 \rtimes \Z_2$, with the generator of $\Z_2$, say $a$, acting on $\Z_3$ by inversion, embeds in a free profinite group, so the centralizer of $a^2$ contains this projective group. 

\begin{theorem}[Description of centralisers] \label{thm: centralisers structure}
Let $G=G_\Gamma$ be a pro-$\C$ RAAG and let $g_0\in G$. Then there is an element $g$ in the conjugacy class of $g_0$ such that its centraliser is of the form
$$C_G(g)=H_1 \times \cdots \times H_s \times \overline{\langle \link(g)\rangle }$$
where:
\begin{enumerate}
    \item $\alpha(H_i),~ \alpha(H_j),~\link(g)$ are all disjoint for $i\neq j$;
    \item $G_{\alpha(g)}=G_{\alpha(H_1)}\times \cdots \times G_{\alpha(H_s)}$;
    \item $H_i$ are projective  pro-$\C$ groups;
    \item if $G$ is pro-$p$, $H_i=\overline{\langle h_i \rangle}$ \; and $g= h_1^{k_1} \cdots h_s^{k_s}$, for some $k_i\in \mathbb Z_p$.
\end{enumerate}
\end{theorem}

\begin{proof}
We begin the proof with some reductions.

If $g$ is trivial, then $V(\Gamma)=\link(g)$ and the result holds trivially, so we further assume $g\neq 1$.

Among the conjugates of $g_0$, we choose an element $g$ of minimal support among its conjugates, so that by Lemma \ref{lem: centralisers} $C_G(g)$ is contained in the standard subgroup generated by $\link(g)\cup \alpha(g)$. Hence, we can assume that $V(\Gamma)=\link(g)\cup \alpha(g)$. In this case, we have from Theorem \ref{thm: directly decomposable} that $G=G_{\alpha(g)}\times G_{\tiny{\link(g)}}$.
Clearly $G_{\tiny{\link(g)}}\leq C_{G}(g)$, so it suffices studying the centraliser in the standard subgroup $G_{\alpha(g)}$ and then
$$C_{G}(g)=C_{G_{\alpha(g)}}(g) \times G_{\tiny{\link}(g)}.$$
We further assume that $\alpha(g)=V(\Gamma)$. If $G$ is decomposable as a direct product $G_\Gamma=G_1\times \cdots \times G_s$, then $g=g_1\times \dots \times g_s$ for $g_i\in G_i$, $i\in \{1,\ldots,s\}$, and the centraliser $C_{G}(g)$ decomposes as $C_{G}(g)=C_{G_1}(g_1)\times \dots \times C_{G_s}(g_s)$. Moreover, by Theorem \ref{thm: directly decomposable}, each $G_i$ is a standard subgroup. As $g$ was chosen to be an element of minimal support among its conjugates, each $g_i$ has also minimal support among its conjugates, so we have reduced the problem to studying centralisers when $G=G_{\alpha(g)}$ is directly indecomposable.

Our goal is to show that if $G=G_{\alpha(g)}$ is directly indecomposable, then $C_G(g)$ is a projective group.
Fix any vertex $v$ of $\Gamma$ and denote by $G_0=G_{\Gamma \smallsetminus \{v\}}$ and by $A=G_{\tiny{\link(v)}}$. Consider the decomposition as an HNN extension
$$G=HNN(G_0, A, id);$$
as described in Remark \ref{rem: hyperbolic splitting}, the action of $g$ on the standard pro-$\C$ tree $T$ associated with this splitting is hyperbolic.\\
We first claim that no nontrivial element $h\in C_G(g)$ is contained in a conjugate of $A$. Indeed, take any element $h\in C_G(g)$ and assume that $t\in G$ is an element such that $h^t$ has minimal support among the $G$-conjugates of $h$, so that by Lemma \ref{lem: minimal standard subgroups} we have $\alpha(h^t)\subseteq \Gamma \smallsetminus \{v\}$. Then $h^t\in C_G(g^t)$ and by Lemma \ref{lem: centralisers} we have $\alpha(g^t)\subseteq \alpha(h^t)\times \link(h^t)$. By Lemma \ref{lem: minimal standard subgroups}, $V(\Gamma)=\alpha(g)\subseteq \alpha(g^t)$, but then $G=G_{\alpha(h^t)}\times G_{\tiny{\link(h^t)}}$. As $G$ is directly indecomposable and $\alpha(h^t)$ is a proper subset of $V(G)$, $h$ must be trivial.

Let $T_g$ be the minimal $g$-invariant subtree of $g$. From the preceding paragraph we deduce that $C_G(g)$ acts faithfully on $T_g$ and so by \cite[Lemma 4.2.6]{RibesBook} is projective. 
In particular,  if  $G$ is a pro-$p$ group, then $C_G(g)$ must be isomorphic to $\Z_p$.
\end{proof}

In general, centralisers and normalisers of elements in pro-$\C$ groups do not need to coincide. Indeed, we already noticed that in a free profinite group, there are subgroups of the form $\mathbb{Z}_3 \rtimes \mathbb{Z}_2$ with $a$, the generator of $\mathbb{Z}_2$, acting by inversion on $\mathbb{Z}_3$. Then $a$ is in the normaliser of $\mathbb{Z}_3$, but it does not centralize it.

We next show that in the case of pro-$p$ groups, the situation is much tamer, namely centralisers and normalisers coincide; they are finitely generated and in fact, they are virtual retracts.

\begin{corollary}
    For each element $g$ of a pro-$p$ RAAG $G_\Gamma$ we have that $N_G(\langle g \rangle )=C_G(\langle g \rangle )$.
\end{corollary}

\begin{proof}
    Let $h\in N_G(\langle g \rangle )$. By Theorem \ref{thm: two generated subgroups}, the solvable subgroup $\langle g,h \rangle$ must be abelian for any $h\in N_G(\langle g \rangle)$, so $h\in C_G(g)$.
\end{proof}

Our next goal is to prove that centralisers of elements are virtual retracts in pro-$p$ groups. 

\begin{lemma}[Procyclic groups are virtual retracts] \label{lem: cyclic are vretracts}
Let $G$ be a pro-$p$ group acting without fixed points on a pro-$p$ tree $T$. Assume that $H$ is a procyclic subgroup, generated by a hyperbolic element $g$. Then $H$ is a virtual retract of $G$.
\end{lemma}

\begin{proof}
For each subgroup $K$ of $G$, define 
$$\widetilde{K}=\overline{\langle K\cap G_t|t\in V(T)\rangle}$$
to be the subgroup generated by the intersections of $K$ with all vertex stabilisers. Notice that $\widetilde{H}=1$ since $H$ is procyclic generated by a hyperbolic element. As $H$ is closed, it is the intersection of all open subgroups $\{U_i,i\in I\}$ containing it and then also $\bigcap_{i\in I} \widetilde{U_i}=\widetilde{H}=1 $. This implies that there must be an open $U\trianglelefteq_o G$ such that $g\notin \widetilde{U}$. But $U/\widetilde{U}$ is a free pro-$p$ group by \cite[Corollary 3.6]{NewHorizons} hence by the profinite version of Marshall Hall Theorem (see \cite[Theorem 9.1.19]{RZBook}) the procyclic subgroup $H\widetilde{U}/\widetilde{U}$ of $U/\widetilde{U}$ is a free factor of a finite index subgroup of $U/\widetilde{U}$. As $H\cap \widetilde{U}$ is trivial, we can lift the retracts to $U$ and we have that $H\cong H\widetilde{U}/\widetilde{U}$ is a virtual retract of $U/\widetilde{U}$, hence a virtual retract of $G$ too.
\end{proof}

\begin{theorem}[Centralisers are virtual retracts]
Let $G=G_\Gamma$ be a pro-$p$ RAAG and let $H$ be the centraliser of an element $h$. Then $H$ is a virtual retract of $G$.
\end{theorem}

\begin{proof}
We can assume that $h$ has minimal support among its conjugates. By Theorem \ref{thm: centralisers structure}, $H$ is contained in the standard subgroup generated by $\alpha(h)\cup \link(h)$, which is a retract of $G$. As a virtual retract of a standard subgroup of $G$ can be lifted to a virtual retract of $G$, we restrict to the case that $V(\Gamma)=\alpha(h)\cup \link(h)$. Suppose then that $G=G_1\times \cdots \times G_k \times G_{\tiny{\link(h)}}$ is the direct product decomposition of the standard subgroup $G$. By Theorem \ref{thm: centralisers structure}, $H=H_1 \times \dots \times H_k \times G_{\tiny{\link(h)}}$, where $H_i$ is a procyclic subgroup of $G_i$ for each $i\in \{1,\ldots,k\}$. By Lemma \ref{lem: cyclic are vretracts} and Remark \ref{rem: hyperbolic splitting}, every $H_i$ is a virtual retract of $G_i$, hence the direct product of all of them is a virtual retract of $G$ and the result follows.
\end{proof}

\section{Abelian splittings of profinite RAAGs}

The main goal of this section is to describe when and how a pro-$\C$ RAAGs splits over a pro-$\C$ abelian group. We begin with two auxiliary lemmas.

\begin{lemma}[Non-elliptic standard subgroups]\label{lem: generators in different vertices}
Let $G=G_\Gamma$ be a pro-$\mathcal{C}$ RAAG associated with a connected graph $\Gamma$. Suppose that $G$ acts on a pro-$\mathcal{C}$ tree $T$ without a global fixed point, and that all canonical generators are elliptic. Then there exist two canonical generators $v,w\in V(\Gamma)$ such that $(v,w)\notin E(\Gamma)$ and $\langle v,w \rangle$ does not stabilize any vertex of $T$.
\end{lemma}

\begin{proof}
Let $T^v$ be the subtree of fixed points of a canonical generator $v$. If, by contradiction, $T^v\cap T^w\neq \varnothing$ for each couple of canonical generators $v,w \in V(\Gamma)$, by Lemma \ref{lem: Helly}, there is a point contained in $\bigcap_{v\in V(\Gamma)} T^v$ fixed by all the generators and so fixed by $G$, contradicting the hypothesis. This implies that there are at least two vertices $v,w\in V(\Gamma)$ such that $\langle v,w \rangle$ does not stabilize any vertex of $T$. Notice that such vertices cannot be adjacent by Lemma \ref{lem: culler vogtmann} (2).
\end{proof}

\begin{lemma}[Hyperbolic generators] \label{lem: Groves Hull}
Let $G_\Gamma$ be a pro-$\mathcal{C}$ RAAG over a connected graph $\Gamma$ acting on a pro-$\mathcal{C}$ tree $T$ with abelian edge stabilisers. Suppose that a canonical generator $v\in G_\Gamma$ is hyperbolic, then:
\begin{enumerate}
   \item $\mbox{Star}(v)$ is a complete graph;
    \item either $V(\Gamma)= \mbox{Star}(v)$ or the set $S:=\{u\in \link(v)\mid \mbox{Star}(u) \mbox{ is not a complete graph}\}$ separates $\mbox{Star}(v)\smallsetminus S$ and $\Gamma \smallsetminus \mbox{Star}(v)$;
    \item the standard subgroup generated by $S$ stabilizes an edge. 
\end{enumerate}
\end{lemma}

\begin{proof}
\begin{enumerate} 
    \item If there exists a single vertex adjacent to $v$, then the result holds. Suppose then that there exist two distinct vertices $w_1,w_2\in \link(v)$.\\
    For each canonical generator $w$ commuting with $v$ we can restrict to the minimal subtree $T_{\langle v, w \rangle}$ on which the abelian subgroup ${\langle v, w \rangle}$ acts. By Theorems \ref{thm: RZ structure amalgam} and \ref{thm: RZ structure HNN}, the group $\langle v,w \rangle$ is a procyclic extension of the kernel of this action and since $\langle v,w \rangle$ is abelian of rank 2,  there exists an element $g=ab$ with $a \in {\langle v \rangle}\leq G$ and $b \in {\langle w \rangle}\leq G$ with $b\neq 1$ (as $v$ is hyperbolic) in the kernel of the action, i.e. $g$ fixes pointwise the minimal subtree $T_{\langle v, w \rangle}$.
    Pick now two elements $g_i=a_ib_i$ with $i\in \{1,2\}$ for $a_1,a_2\in \overline{\langle v \rangle}$, $b_1\in \overline{\langle w_1 \rangle}$, $b_2\in \overline{\langle w_2 \rangle}$ such that $b_1,b_2$ are not trivial and such that $g_1,g_2$ stabilize pointwise $T_v$. By hypothesis $g_1,g_2$ are contained in the abelian stabilisers of the edges of $T_v$. Let $K=\langle g_1,g_2\rangle$ and let $f$ be the retraction of $G$ onto the standard subgroup generated by $w_1,w_2$. The image $f(K)\leq G_{\{w_1,w_2\}}$ is an abelian subgroup that contains $b_1$ and $b_2$.
    The element $b_1$ is in the centraliser of $b_2$ and they are both with minimal support among their conjugates, so applying Lemma \ref{lem: centralisers} this can happen only if $w_1 \in \link(b_2)=\link(w_2)$, so $w_1,w_2$ are adjacent and $\mbox{Star}(v)$ is a complete graph.
    \item Suppose $V(\Gamma) \neq \mbox{Star}(v)$, as $\Gamma$ is connected we have that
    $$S=\{u\in \link(v)\mid Star(u) \mbox{ is not a complete graph}\}$$
    is non-empty. It is immediate to see that $S$ separates the subgraphs generated by $\mbox{Star}(v)\smallsetminus S$ and $\Gamma \smallsetminus \mbox{Star}(v)$ because, as $\link(v)$ is a complete graph by Item (1), each vertex in $\mbox{Star}(v)\smallsetminus S$ is connected only to vertices in $\mbox{Star}(v)$.
    \item By Lemma \ref{lem: culler vogtmann}(1), each vertex of $S$ fixes the subtree $T_v$, which contains at least an edge because $v$ is hyperbolic. By Item (1), the action on $T$ of any element of $S$ is elliptic, and hence $T_v$ is fixed pointwise by $S$.
\end{enumerate}
\end{proof}

We are now ready to prove the main Theorem.

\begin{theorem}[Graph characterisation of splittings over abelian edge groups] \label{thm: abelian splitting of RAAGs}
Let $G=G_\Gamma$ be a pro-$\mathcal{C}$ RAAG associated with a connected graph $\Gamma$. Then $G$ acts on a pro-$\mathcal{C}$ tree with abelian edge stabilisers without a global fixed point if and only if either $\Gamma$ is a complete graph or $\Gamma$ has a disconnecting complete graph.\\
In the second case, there exists a disconnecting complete graph whose standard subgroup is contained in one edge stabiliser of $T$. 
\end{theorem}

\begin{proof}
The case when $\Gamma$ is a complete graph is clear: indeed, denoting by $\pi=\pi(\mathcal{C})$, the pro-$\mathcal{C}$ RAAG $G_\Gamma$ is isomorphic to $\mathbb{Z}_{\pi}^{n}$, that splits as an HNN-extension $HNN(\mathbb{Z}_{\pi}^{n-1},\mathbb{Z}_{\pi}^{n-1}, id)$.
Similarly, if there is a complete graph $K$ that disconnects $\Gamma$, i.e. $\Gamma \smallsetminus K = \Gamma_1 \cup \Gamma_2$ with $\Gamma_1,\Gamma_2$ disjoint subgraphs, then $G$ splits as
$$G_\Gamma= G_{\Gamma_1\cup K} \amalg_{G_K} G_{\Gamma_2 \cup K}$$
and so $G$ acts on the standard pro-$\mathcal{C}$ tree associated with this splitting.\\
Suppose now that $G$ acts on a pro-${\mathcal{C}}$ tree $T$ with abelian edge stabilisers. If there exists a hyperbolic canonical generator $v$ of $G$, by Lemma \ref{lem: Groves Hull}, we know that either $\Gamma$ is complete or the set $S:=\{u\in \link(v)~| ~\mbox{Star}(u) \mbox{ is not a complete graph}\}$ is a disconnecting complete graph contained in an edge stabiliser. We are left to the case when each canonical generator of $G$ is elliptic.

By Lemma \ref{lem: generators in different vertices}, there exist two vertices $v,w \in V(\Gamma)$ such that no vertex of $T$ is stabilized by both $v$ and $w$. As each canonical generator acts elliptically, let $t_v,t_w$ be two vertices of $T$ stabilized by $v$ and $w$ respectively. Let $S=[t_v,t_w]$ be the geodesic between these two vertices in $T$.\\
By Lemma \ref{lem: generators in different vertices}, $S$ contains at least one edge, and moreover there exists at least one edge of $S$ that is not stabilised either by $v$ or by $w$, as by collapsing the subtrees $S\cap T_v$ and $S\cap T_w$ to a point (noticing that we chose $v,w$ such that $T_v\cap T_w=\varnothing$), we would otherwise have $S$ to be disconnected.
Define $K$ as a maximal (by the number of vertices contained) complete subgraph of $\Gamma$ contained in an edge group of $S$ that is not stabilized by either $v$ or $w$, say $e\in E(T)$. If $K$ is empty, let $e$ be any edge of $S$, not stabilized by $v$ or $w$. It is important to notice that even if $S$ might contain infinitely many edges satisfying the properties, $\Gamma$ is finite hence $K$ is well defined. We claim that $K$ is a complete graph of $\Gamma$ that disconnects the vertices $v$ and $w$.\\
Suppose by contradiction that it is not, then we could find a finite path $p=(v,u_1,\ldots,u_k,w)$ in $\Gamma$, such that no vertex of $p$ is contained in $K$. By Lemma \ref{lem: commutation of a path}, there exist some vertices $t_1,\ldots,t_{k+2}$ such that $v$ stabilizes $t_1$, $u_i$ stabilizes the geodesic $S_i=[t_i,t_{i+1}]$ for $i=1,\ldots,k$, and $w$ stabilizes $[t_{k+1},t_{k+2}]$. Set $t_0=t_v$ and $T_{k+3}=t_w$. In this setting, $v$ stabilizes $S_0=[t_0,t_1]$ and $w$ stabilizes $S_{k+1}=[t_{k+1},t_w]$.
Furthermore, the union $S'=\bigcup_{i\in\{0,\ldots,k+1\}}S_i$ of the $S_i$ is a pro-$p$ tree that contains $t_v$ and $t_w$, hence it contains the whole $[t_v,t_w]$. In particular, $S'$ contains $e$, so $e\in [t_j,t_{j+1}]$ for some $j\in\{0,\ldots,k+3\}$. By the choice of $e$, it cannot be stabilized by $v$ or $w$, so there exists a vertex $u_j$ such that $u_j\in G_e$. Now $G_e$ is an abelian pro-$\C$ subgroup of $G$ that contains $u_j$ and each vertex of $K$, but by Theorem \ref{thm: centralisers structure}, this is only possible if $u_j$ is adjacent to every vertex of $K$. By maximality of $K$, $u_j\in K$, but this contradicts the fact that no element of the path $p$ is contained in $K$. If we assumed $K$ to be empty, we have anyway proved that there is a vertex $u_j$ contained in an edge stabiliser of $T$ contradicting that $K=\varnothing$.\\
This proves that the graph $K$ is a disconnecting complete graph contained in an edge stabiliser, as required.
\end{proof}

\section{JSJ decompositions}

\subsection*{Prerequisites}

In the previous section, we have characterised when a pro-$\mathcal{C}$ RAAGs admits a splitting over an abelian subgroup. Our next goal is to describe all the splittings of these groups over abelian subgroups. In the abstract case, the abelian splittings of a finitely generated group are encoded in a construction called the JSJ decomposition of a group.
We develop this theory following the approach of Guirardel and Levitt in \cite{Guirardel}.  We show that it can be naturally extended to the pro-$\C$ world; for additional results and alternative definitions on the theory of JSJ decompositions see the references in \cite{Guirardel}.

\begin{definition}[$\mathcal{A}$-trees]
    For each class of pro-$\C$ groups $\mathcal{A}$ closed for subgroups and conjugation, we define a $\mathcal{A}$-tree $(T,G)$ as a pro-$\C$ tree $T$ with an action of a pro-$\C$ group $G$ such that each edge stabiliser is a group in the class $\mathcal{A}$.
\end{definition}
We often denote the $\mathcal{A}$-tree as $T$ rather than $(T,G)$ whenever the pro-$\C$ group $G$ acting on it is clear by the context and we will say that an $\mathcal{A}$-tree $(T,G)$ is trivial if $T$ consists of a single vertex stabilized by the whole $G$.

We say that a subgroup $H$ of a pro-$\C$ group $G$ is \emph{universally elliptic} (for actions over $\mathcal{A}$-trees) if the action of $H$ is elliptic over any $\mathcal{A}$-tree $(T,G)$ on which $G$ acts.

\begin{definition}[JSJ decompositions]\label{Definition JSJ Dec}\

\begin{itemize}
    \item An $\mathcal{A}$-tree $(T,G)$ is \emph{universally elliptic} if its edge stabilisers $G_e\leq G$ are universally elliptic for actions on $\mathcal{A}$-trees.
    \item An $\mathcal{A}$-tree $(T,G)$ \emph{dominates} another $\mathcal{A}$-tree $(T',G)$ if the same group $G$ acts on both of them and the action of vertex stabilisers $G_v$, $v\in T$ is elliptic on $T'$ too.
    \item Two $\mathcal{A}$-trees $(T,G)$ and $(T',G)$ are \emph{equivalent} if the same pro-$\C$ group $G$ acts on both of them and they dominate each other. An equivalence class of $\mathcal{A}$-trees for this relation is said to be a \emph{deformation space}.
    \item The deformation space of the $\mathcal{A}$-trees that are universally elliptic and that dominate any other universally elliptic $\mathcal{A}$-tree on which $G$ acts is the \emph{JSJ deformation space} and its elements are called the \emph{JSJ tree decompositions}.
\end{itemize}

\end{definition}

Notice that the deformation space is unique, but there might be many non-isomorphic tree decompositions of a pro-$\C$ group $G$.

\begin{definition}[Rigid and flexible vertices]
A vertex $v$ of a \emph{JSJ}-tree is said to be \emph{rigid} if it is universally elliptic for the action on any $\mathcal{A}$-tree (even if the tree is not universally elliptic) and \emph{flexible} otherwise.
\end{definition}

Notice that if all vertex groups of an $\mathcal{A}$-tree are rigid, then the $\mathcal{A}$-tree is a JSJ tree, but the converse is not true, as the following example shows.

\begin{example}\label{ex: JSJ of complete graphs}
If $G\cong\mathbb{Z}_{\rho}^n$ for $n\geq 2$, $\rho$ an arbitrary set of primes, the abelian JSJ decomposition is trivial.\\
We claim that for each element $g\in G$ we can produce an $\mathcal{A}$-tree $(T,G)$ such that the action of $g$ on $T$ is hyperbolic. Consider a maximal procyclic subgroup $C$ containing $g\in G$. Any generator of a maximal procyclic group can be part of a basis of $\mathbb{Z}_{\rho}^n$, so we can pick a complement $B\cong \mathbb{Z}_{\rho}^{n-1}$ of $C$ in $G$ and write $G=HNN(B,B,id)$ with a generator of $C$ as the stable letter. The standard pro-$\C$ tree associated with this pro-$\C$ HNN extension is a vertex with a single loop and $g$ is hyperbolic by construction. This proves that no edge group can be universally elliptic, hence there exists a single universally elliptic $\mathcal{A}$-tree $(T,G)$ on which $G$ acts, which is a tree $T$ with a single point. This is the JSJ decomposition of $G$, which has a single flexible vertex.
\end{example}

Sometimes is convenient to study \emph{relative} JSJ decompositions, which are defined as follows.
\begin{definition}[Relative JSJ Decompositions]
    Let $\mathcal{H}$ be an arbitrary family of subgroups of a pro-$\C$ group $G$. An $\mathcal{A}$-tree $(T,G)$ is an $(\mathcal{A,H})$-tree if all the subgroups in the class $\mathcal{H}$ are elliptic. An $(\mathcal{A,H})$-tree is an \emph{$(\mathcal{A,H})$-JSJ decomposition} if it is universally elliptic for actions on $(\mathcal{A},\mathcal{H})$-trees and it dominates every other universally elliptic $(\mathcal{A,H})$-tree.
\end{definition}

We now turn our attention to the study of the JSJ-decomposition of a pro-$\C$ RAAG over abelian groups. Let $G=G_\Gamma$ be a pro-$\C$ RAAG over a finite connected graph $\Gamma$. From here on, we assume $\mathcal{A}$ to be the class of abelian pro-$\C$ subgroups of $G$ and $\mathcal{H}$ to be the class of procyclic groups generated by canonical generators of $G$.

We first construct by induction a decomposition of $G$ over abelian subgroups relative to $\mathcal{H}$, and prove that it is actually an $(\mathcal{A,H})$-JSJ decomposition. We then refine this decomposition in order to obtain the $\mathcal{A}$-JSJ decomposition of $G$.

As we are interested in splittings over standard subgroups of disconnecting complete graphs, we first need some basic properties of splittings of this type.

\begin{lemma}[Disconnecting complete graphs are elliptic]\label{lem: complete graphs elliptical}
Let $G_\Gamma$ be a pro-$\C$ RAAG over a finite connected graph $\Gamma$ and $K\leq \Gamma$ be a complete subgraph of $\Gamma$.
\begin{enumerate}
    \item If all cyclic subgroups generated by canonical generators in $K$ are universally elliptic for their action on $\mathcal{A}$-trees, then the whole standard subgroup $G_K$ is universally elliptic for its action on $\mathcal{A}$-trees.
    \item If $K$ is a minimal disconnecting complete graph (in the sense that no proper subset of $K$ is a disconnecting complete graph), then the standard subgroup $G_K$ is universally elliptic for its action on $\mathcal{A}$-trees.
    \item If $\mbox{Star}(v)$ is a complete graph for $v\in V(\Gamma)$, then there exists an $\mathcal{A}$-tree on which the action of $v$ is hyperbolic.
\end{enumerate}
\end{lemma}
\begin{proof}\
\begin{enumerate}
    \item Since by assumption $\Gamma$ is connected, this follows as a consequence of Lemma \ref{lem: culler vogtmann} (2).
    \item Assume that $G$ acts on an $\mathcal{A}$-tree $(T,G)$ and suppose that there exists at least one hyperbolic canonical generator $v\in V(K)$. By Lemma \ref{lem: Groves Hull}(1), we have that $\mbox{Star}(v)$ is a complete graph. Since a complete graph does not have any disconnecting subgraphs, it follows that $\Gamma\neq \mbox{Star}(v)$.
    From the minimality of the disconnecting complete graph $K$, we have that the full subgraph $\Gamma'$ generated by $(V(\Gamma)\smallsetminus V(K))\cup \{v\}$ is connected and $v$ is a disconnecting vertex of $\Gamma'$. In particular, there are two vertices $w_1,w_2\in V(\Gamma')$ that are adjacent to $v$ but lie in different connected components of $\Gamma \smallsetminus K$. This contradicts the fact that $\mbox{Star}(v)$ is a complete graph.
    Hence each canonical generator of $K$ must be elliptic and by Item (1) the whole $K$ is elliptic.
    \item It suffices to notice that the standard pro-$\C$ tree associated with the splitting (\ref{eq: splitting with a vertex}) has abelian edge stabilisers because $\link(v)$ is a complete graph.
\end{enumerate}
\end{proof}

We record the following graph theoretical observation.

\begin{lemma}[Disconnecting graphs of components] \label{lem: disconnecting graph of connected component}
    Let $\Gamma$ be a finite connected simplicial graph. Let $K$ be a disconnecting complete subgraph of $\Gamma$ and let $\{\Gamma^i \mid i\in \{1,\ldots, m\}\}$ be the connected components of $\Gamma\smallsetminus K$. 
    
    If $K'$ is a disconnecting subgraph of $\Gamma^j\cup K$ for some $j\in \{1,\ldots, m\}$, then $K'$ is also a disconnecting subgraph of $\Gamma$.
\end{lemma}

\begin{proof}
    Suppose on the contrary that $\Gamma \smallsetminus K'$ is connected. Since $K$ is by assumption a disconnecting subgraph of $\Gamma$, it follows that $K$ is not contained in $K'$ and so $K \setminus K'$ is nonempty. Since $\Gamma \smallsetminus K'$ is connected and $K$ is disconnecting, for each vertex $v$ in $(\Gamma^j\cup K)\smallsetminus K'$ there is a vertex $w(v)$ in $K\smallsetminus K'$ such that $v$ and $w(v)$ are connected by a path inside $(\Gamma^j\cup K)\smallsetminus K'$. As $K$ is complete, there is an edge between any two vertices in $K$. It follows that any pair of vertices $v,v' \in (\Gamma^j\cup K)\smallsetminus K'$ are connected by the path which is the composition of the paths from $v$ to $w(v)$, the edge $(w(v), w(v'))$ and the path from $w(v')$ to $v'$. Since this path is in $(\Gamma^j\cup K)\smallsetminus K'$, we have that $(\Gamma^j\cup K)\smallsetminus K'$ is connected, deriving a contradiction.
\end{proof}

\subsection*{\texorpdfstring{$(\mathcal{A,H})$}{(A,H)}-JSJ decomposition of pro-\texorpdfstring{$\C$}{C} RAAGs}

We first construct the (relative) abelian JSJ decomposition of pro-$\C$ RAAGs under the assumption that all the subgroups in the class $\mathcal{H}=\{\langle v \rangle \mid v\in V(\Gamma)\}$ of procyclic subgroups generated by canonical generators are elliptic.

\begin{theorem}[Relative JSJ-decomposition] \label{thm: A-H JSJ decomposition}
Let $G=G_\Gamma$ be a pro-$\C$ RAAG associated with a connected abstract finite graph $\Gamma$.

There is a (possibly trivial) decomposition of $G$ as a fundamental pro-$\C$ group of a reduced finite tree of pro-$\C$ groups $(\mathcal{G}_\Delta, \Delta)$ with the following properties:

\begin{itemize}
    \item vertex groups of $(\mathcal{G}_\Delta, \Delta)$ are standard subgroups which are either abelian or their underlying graph does not contain any disconnecting complete subgraph;
    \item each edge group of $(\mathcal{G}_\Delta, \Delta)$ is a standard subgroup associated with a disconnecting complete subgraph $K_e$ of $\Gamma$ and, moreover, $K_e$ is a minimal (with respect to inclusion) disconnecting complete graph of a subgraph $\Gamma'$ of $\Gamma$.
\end{itemize}

Furthermore, the standard pro-$\C$ tree associated with this decomposition is an $(\mathcal{A,H})$-JSJ tree decomposition $(T_\Delta,G)$ of $G$.
\end{theorem}

\begin{proof}
We prove the statements by induction on the number of generators of the pro-$\C$ RAAG.

Assume first that $\Gamma$ has one vertex, i.e. $G=\mathbb{Z}_{\pi(\C)}$. In this case, we consider the decomposition as a fundamental group of a graph of groups to be trivial, so $\Delta$ is a point and the associated group is $\mathbb{Z}_{\pi(\C)}$. This decomposition satisfies the required conditions. Furthermore, since $G$ is a standard subgroup, by assumption it is elliptic and so the $(\mathcal{A,H})$-JSJ decomposition of $G$ is trivial and agrees with the decomposition as a fundamental group of a graph of groups.

Assume that we have already established the decomposition of every pro-$\C$ RAAG whose underlying graph has at most $n-1$ vertices as a fundamental group of a graph of groups and that we have proved that the $(\mathcal{A,H})$-JSJ decomposition of $G$ is determined by the group decomposition as a fundamental group of a graph of pro-$\C$ groups satisfying the properties of the theorem.

Let now $\Gamma$ be a connected graph with $n$ vertices, $n\ge 2$. Suppose first that $\Gamma$ does not have any disconnecting complete subgraph. In this case, we consider the decomposition as a fundamental group of a graph of groups to be trivial and so $\Delta$ has one vertex with corresponding group $G$. This decomposition satisfies the requirements. If $\Gamma$ is a complete graph, then $G \simeq \mathbb Z_{\pi(\C)}^n$. Since by assumption, each canonical generator is elliptic, then by Lemma \ref{lem: complete graphs elliptical}, the group $G$ stabilizes a point, and hence the $(\mathcal{A,H})$-JSJ decomposition is trivial and coincides with the decomposition of $G$ as the fundamental group of a graph of groups. If $\Gamma$ is not complete and does not have any disconnecting complete subgraph, then by Theorem \ref{thm: abelian splitting of RAAGs}, $G$  cannot act non-trivially on an $\mathcal{A}$-tree, so the $(\mathcal{A,H})$-JSJ decomposition is again trivial.

Suppose now that $\Gamma$ has a disconnecting complete graph. Let $K$ be a disconnecting complete graph such that $|V(K)|$ is minimal among disconnecting complete graphs.

We first construct a splitting of $G$ as an amalgamated free product over the standard subgroup $G_{K}$. Assume that $\Gamma \smallsetminus K$ has $m\geq 2$ nontrivial connected components $\Gamma^{i}$, for $i\in \{1,\ldots,m\}$. In this case, we consider the splitting of $G$ as a pro-$\C$ amalgamated product of the form
$$G=\coprod_{i=1}^m{}_{_{G_{K}}}G_{K\cup \Gamma^i}.$$
By \cite[Theorem 6.5.2]{RibesBook}, this decomposition corresponds to the pro-$\C$ fundamental group of a tree of groups $(\overline{\mathcal{G}_\Delta},\overline{\Delta})$ with $m$ vertices $V(\overline{\Delta})=\{x_1,\ldots,x_m\}$, whose vertex groups $\{G_{K\cup \Gamma^i} \mid i\in \{1,\ldots,m\}\}$ respectively, and with all edges of $E(\overline{\Delta})$ stabilised by $G_K$. Since $K$ is a complete graph, $G_K$ is a pro-$\C$ abelian subgroup and hence this decomposition is an $\mathcal A$-decomposition of $G$.
Notice that if $m>2$, the underlying tree $\overline{\Delta}$ is not unique. Indeed any tree with $m$ vertices provides the same fundamental group $G$ since all edge groups coincide. Without loss of generality, we choose the underlying graph $\overline{\Delta}$ to be a path consisting of $m$ points and $m-1$ edges, vertex groups $G_{K \cup \Gamma^i}$ and edge groups $G_K$ (with the natural embeddings). By construction, the graph of pro-$\C$ groups $(\overline{\mathcal{G}_\Delta},\overline{\Delta})$ has $G$ as its pro-$\C$ fundamental group.

By the induction hypothesis, for each $i\in \{1,\ldots,m\}$ each vertex group $G_{K\cup \Gamma^i}$ has a decomposition as a fundamental group of a tree of pro-$\C$ groups $(\mathcal{G}_{\Delta_i}, \Delta_i)$ as in the statement and this decomposition determines an $(\mathcal{A,H})$-JSJ tree decomposition.

For each $i\in \{1,\ldots,m\}$, by Lemma \ref{lem: complete graphs elliptical}, the action of the group $G_K$ is elliptic on any $(\mathcal{A},\mathcal{H})$-tree $(T_{K\cup \Gamma^i, } G_{K\cup \Gamma^i})$ and so $G_K$ is contained in a vertex stabiliser of $T_{K\cup \Gamma^i}$. Hence, a conjugate of $G_K$ is contained in a vertex group of the graph of groups $(\mathcal{G}_{\Delta_i},\Delta_i)$, namely $v_i \in \Delta_i$. By Lemma \ref{lem: minimal standard subgroups}, if a conjugate of a canonical generator is contained in a standard subgroup, then the standard subgroup contains the generator. As each vertex group of $(\mathcal{G}_{\Delta_i},\Delta_i)$ is a standard subgroup by induction, $G_K$ itself is contained in the vertex group $\mathcal{G}_{\Delta_i}(v_i)$.

We construct a tree of groups $(\mathcal{G}_\Delta, \Delta)$ in the following way. Define $V(\Delta)=\bigcup_{i=1}^m V(\Delta_i)$ and
$$ 
E(\Delta)=\{E(\Delta_i), (v_j,v_\ell) \mid i\in \{1,\ldots,m\}, (v_j,v_\ell)\in E(\overline{\Delta})
\}.
$$
For each $w \in V(\Delta)$ there is $i\in \{1, \dots, m\}$ such that $w\in \Delta_i$ and we define the group $\mathcal{G}_\Delta(w)$ of $\mathcal{G}_\Delta$ to be $\mathcal{G}_\Delta(w)=\mathcal{G}_{\Delta_i}(w)$. Similarly, if $e$ is an edge of $\Gamma$ such that $e\in E(\Delta_i)$, then the corresponding group (and vertex embeddings) are induced from $\Delta_i$. If $e=(v_j,v_\ell)$, we define $\mathcal{G}_\Delta(\delta)=G_K$ (with the natural embeddings). This graph of groups is well-defined as each edge group embeds in the adjacent vertex groups and its pro-$\C$ fundamental group is exactly $G$ by construction (as the fundamental group of the graphs $(\mathcal G_{\Delta_i}, \Delta_i)$ are the standard subgroups $G_{K \cup \Gamma_i}$). 

This graph of groups is reduced. Indeed by induction, edge groups of $(\mathcal{G}_{\Delta_i},\Delta_i)$ do not coincide with the adjacent vertex groups. We next show that $G_K$ cannot coincide with any vertex group in $\mathcal{G}_\Delta(v_i)$. Indeed, by definition, $\Gamma_i$ is a nontrivial connected component of $\Gamma \setminus K$ and so $G_{K\cup \Gamma^i}\neq G_K$ and, in particular, if the decomposition of $G_{K\cup \Gamma^i}$ is trivial, then the unique vertex group does not coincide with $G_K$. Assume next that $G_{K\cup \Gamma^i}$ has a nontrivial decomposition satisfying the conditions of the statement and suppose by contradiction that there is a vertex group of $(\mathcal{G}_{\Delta_i},\Delta_i)$ equal to $G_K$. In particular, since the graph $K\cup \Gamma^i$ is connected and edge groups are standard subgroups of complete disconnecting subgraphs of $K\cup \Gamma^i$, there would be disconnecting subgraph $K'$ of $K\cup \Gamma^i$ contained in $K$. By Lemma \ref{lem: disconnecting graph of connected component}, $K'$ would also be a disconnecting complete subgraph of $\Gamma$, contradicting the minimality of $K$. Hence, we have shown that the graph of groups is reduced.

We next show that the decomposition as a fundamental group of a graph of groups satisfies the required properties. Indeed, by the inductive hypothesis on $\Delta_i$, we have that each vertex group of $\mathcal{G}_\Delta$ is a standard subgroup which is either abelian or the underlying graph does not have disconnecting complete subgraphs; and the edge groups of $\mathcal{G}_\Delta$ are either $G_K$ or, by induction, they are standard subgroups associated with disconnecting complete subgraphs of a certain ${K\cup \Gamma^i}$, which are also disconnecting subgraphs for $\Gamma$ by Lemma \ref{lem: disconnecting graph of connected component}. In the former case, by our choice $G_K$ is the standard subgroup of a minimal complete disconnecting subgraph of $\Gamma$. In the latter case, the induction hypothesis assures that the associated disconnecting complete subgraph is minimal for a subgraph of ${K\cup \Gamma^i}$, which is also a subgraph of $\Gamma$.

Finally, we are left to check that the standard tree $(T_\Delta,G)$ associated with the decomposition as a fundamental group of a graph of groups given for $G$ is an $(\mathcal{A,H})$-JSJ tree. The tree $(T_\Delta,G)$ is universally elliptic since each edge stabiliser is either universally elliptic by the induction hypothesis or it is a conjugate of $G_K$ and since $G_K$ is abelian and generated by universally elliptic elements, by Lemma \ref{lem: complete graphs elliptical} (2), $G_K$ acts universally elliptic on any $\mathcal{A,H}$-tree.

In order to prove that $(T_\Delta,G)$ dominates any other $(\mathcal{A,H})$-tree $(T',G)$, consider a vertex stabiliser $H\leq G$ given by the decomposition of $G$ as a fundamental group of graphs of groups. By construction, $H$ is either a standard subgroup associated with a complete graph, with an elliptic action on any $(\mathcal{A,H})$-tree by Lemma \ref{lem: complete graphs elliptical} (2), or it is a standard subgroup associated with a graph without disconnecting complete subgraphs, which is also elliptic for the action on any $(\mathcal{A,H})$-tree by Theorem \ref{thm: abelian splitting of RAAGs}.

Therefore, the $(\mathcal{A,H})$-tree is a JSJ-tree decomposition of $G$.
\end{proof}

\subsection*{\texorpdfstring{$\mathcal{A}$}{A}-JSJ decomposition of pro-\texorpdfstring{$\C$}{C} RAAGs}

In order to obtain the general $\mathcal{A}$-JSJ decomposition, we must further refine the $(\mathcal{A,H})$-JSJ decomposition described in Theorem \ref{thm: A-H JSJ decomposition}.

\begin{definition}[Hanging vertex]
We say that a vertex $v$ of $\Gamma$ is a \emph{hanging vertex} if $\mbox{Star}(v)$ is a complete graph and for each $w\in \link(v)$, $\mbox{Star}(w)$ is not a complete graph.   
\end{definition}

\begin{theorem}[Abelian JSJ decomposition] \label{thm: A JSJ decomposition}
Let $G=G_\Gamma$ be a pro-$\C$ RAAG associated with a connected abstract finite graph $\Gamma$.

There is a (possibly trivial) decomposition of $G$ as a fundamental pro-$\C$ group of a reduced finite graph of pro-$\C$ groups $(\mathcal{G}_\Theta, \Theta)$ with the following properties:

\begin{itemize}
    \item the underlying graph $\Theta$ is either a tree or a tree with loops;
    \item vertex groups of $(\mathcal{G}_\Theta, \Theta)$ are standard subgroups which are either abelian or their underlying graph does not contain any disconnecting complete graph;
    \item each edge group of $(\mathcal{G}_\Theta, \Theta)$ is a standard subgroup associated with a disconnecting complete subgraph of $\Gamma$;
    \item hanging vertices do not belong to any vertex group.
\end{itemize}

Furthermore, the standard pro-$\C$ tree associated with this decomposition is an $\mathcal{A}$-JSJ tree decomposition $(T_\Theta,G)$ of $G$.
\end{theorem}

\begin{proof}
Suppose first that $|V(\Gamma)|=1$, so $G=\mathbb{Z}_{\pi(\C)}$. In this case define a graph $\Theta$ consisting of a single vertex with a loop, so that $V(\Theta)=\{v_\Theta\}$, $E(\Theta)=\{e_\Theta\}$ with $d_0(e_\Theta)=d_1(e_\Theta)=v_\Theta$. Define the corresponding vertex and edge group as $\mathcal{G}_\Theta(e_\Theta)=\mathcal{G}_\Theta(v_\Theta)=1$ (with the natural embedding).
The graph of groups $(\mathcal{G}_\Theta, \Theta)$ satisfies the required properties and the associated pro-$\C$ tree is the $\mathcal{A}$-JSJ decomposition of $G$ because the trivial element is always elliptic. 

Assume now that $|V(\Gamma)|\geq 2$. If $\Gamma$ does not have disconnecting complete graphs, then the trivial decomposition, with $\Theta$ consisting of a single vertex with corresponding group $G$, satisfies the requirements. If $\Gamma$ is a complete graph, then the associated tree decomposition $(T_\Theta,G)$, which is trivial, is the $\mathcal{A}$-JSJ decomposition, see discussion in Example \ref{ex: JSJ of complete graphs}. Similarly, if $\Gamma$ is not complete and it has no disconnecting complete subgraph, then the $\mathcal{A}$-JSJ decomposition is trivial by Theorem \ref{thm: abelian splitting of RAAGs}.

In the case when $\Gamma$ has disconnecting complete subgraphs, we first consider the graph of pro-$\C$ groups $(\mathcal{G}_\Delta, \Delta)$ as described in Theorem \ref{thm: A-H JSJ decomposition}.

Let $HV(\Gamma)\subset V(\Gamma)$ be the set of hanging vertices of $\Gamma$. We claim that, for each $v\in HV(\Gamma)$, the standard subgroup $G_{\tiny{\mbox{Star}(v)}}$ coincides with an abelian vertex group of $(\mathcal{G}_\Delta, \Delta)$. Since by definition $\mbox{Star}(v)$ is complete, $G_{\tiny{\mbox{Star}(v)}}$ is abelian and so by Lemma \ref{lem: complete graphs elliptical} the action of this subgroup on $T_\Delta$ is elliptic and therefore a conjugate of this subgroup is contained in at least one vertex group of $(\mathcal{G}_\Delta, \Delta)$. As each of these vertex groups is a standard subgroup, by Lemma \ref{lem: minimal standard subgroups} $G_{\tiny{\mbox{Star}(v)}}$ itself is contained in them. Notice that if $\mbox{Star}(v)$ disconnects a graph, then so does $\link(v)$, and similarly, if $\mbox{Star}(v)$ is contained in a complete disconnecting subgraph $K$, then $K\setminus \{v\}$ is also a disconnecting subgraph. Since edge groups are minimal complete disconnecting subgraphs of a subgraph of $\Gamma$, see Theorem \ref{thm: A-H JSJ decomposition}, from the latter observation we have that $G_{\tiny{\mbox{Star}(v)}}$ cannot be contained in any edge group of the graph of groups $(\mathcal G_{\Delta}, \Delta)$ and so $G_{\tiny{\mbox{Star}(v)}}$ is contained in a unique vertex group, namely the vertex group $G_{\Gamma'}$, which by Theorem \ref{thm: A-H JSJ decomposition} is a standard subgroup associated with some subgraph $\Gamma' < \Gamma$. If $G_{\tiny{\mbox{Star}(v)}} \lneq G_{\Gamma'}$, then there is no edge between vertices in $\Gamma' \smallsetminus \mbox{Star}(v) \ne \varnothing$ and $v$, and so $G_{\Gamma'}$ is not abelian and $\link(v)$ is a disconnecting subgraph of $\Gamma'$, contradicting the vertex group description of Theorem \ref{thm: A-H JSJ decomposition}. Therefore, we have that $G_{\tiny{\mbox{Star}(v)}}$ is precisely the vertex group.

Similarly, $v$ cannot be contained in any edge group of $(\mathcal{G}_\Delta, \Delta)$ because such groups are minimal disconnecting complete graphs $K$, and $K\smallsetminus v$ would also be a disconnecting complete graph. For this reason, for each $v\in HV(\Gamma)$, there exists only a single vertex $d_v\in V(\Delta)$ such that $\langle v \rangle \leq \mathcal{G}_\Delta(d_v)$. Notice that, as $\mathcal{G}_\Delta(d_v)$ is abelian, it is immediate from the definition of hanging vertex that $v$ is the only hanging vertex contained in $\mathcal{G}_\Delta(d_v)$, so $d_{v_1}\neq d_{v_2}$ for each distinct $v_1, v_2 \in HV(\Gamma)$.

We define a graph of groups $(\mathcal{G}_{\Delta_0},\Delta_0)$ in the following way. We define $V(\Delta_0)=V(\Delta)$ and $E(\Delta_0)=E(\Delta)\cup \{e_v\mid v\in HV(\Gamma)\}$, where $d_0(e_v)=d_1(e_v)=d_v$. For $w\in V(\Delta_0)$, if $w=d_v$ for some $v\in HV(\Gamma)$, then we set $\mathcal{G}_{\Delta_0}(d_v)=\mathcal{G}_{\Delta_0}(e_v)=G_{\tiny{\link(v)}}$ (the embeddings from the edge groups to the vertex groups are the identity) and otherwise, we set $\mathcal{G}_{\Delta_0}(w)=\mathcal{G}_{\Delta}(w)$ for $w\in V(\Delta_0), w\ne d_v$, $v$ a handing vertex and $\mathcal{G}_{\Delta_0}(e)=\mathcal{G}_{\Delta}(e)$ for $e\in E(\Delta)$. As we observed above, since hanging vertices do not belong to any edge group, the embeddings from edges groups to vertex groups in $(\mathcal G_{\Delta}, \Delta)$ also define embeddings in $(\mathcal G_{\Delta_0}, \Delta_0)$. 

The graph of pro-$\C$ groups $(\mathcal{G}_{\Delta_0},\Delta_0)$ may not be reduced, so we define $(\mathcal{G}_\Theta, \Theta)$ as the reduced graph of groups obtained from $(\mathcal{G}_{\Delta_0},\Delta_0)$. By construction, the graph of groups $(\mathcal{G}_\Theta, \Theta)$ is reduced. The underlying graph $\Theta$ is obtained from $\Delta_0$ by collapsing some edges and in turn, the graph $\Delta_0$ is obtained by adding loops to the tree $\Delta$ and therefore $\Theta$ is a tree with loops. Vertex and edge groups are either equal to $G_{\tiny{\link(v)}}$ for some $v\in HV(\Gamma)$ or they inherit the structure of vertex and edge groups of $(\mathcal{G}_\Delta, \Delta)$. No hanging vertex can be contained in any vertex group by construction. As the pro-$\C$ fundamental group of each vertex with loop $(\mathcal{G}_{\Delta_0}, \{d_v,e_v\})$ is exactly the pro-$\C$ fundamental group of $(\mathcal{G}_\Delta, \{d_v\})$, the pro-$\C$ fundamental group of $(\mathcal{G}_\Theta, \Theta)$ is also $G$. Therefore, the decomposition of $G$ as a fundamental group of graph of groups satisfies the requirements of the statement.

In order to conclude, we have to check that the standard tree $(T_\Theta,G)$ associated with the decomposition given for $G$ as a fundamental group of a graph of groups is an $\mathcal{A}$-JSJ tree. Edge stabilisers are either conjugates of a standard subgroup $G_{\tiny{\link(v)}}$ for some $v\in HV(\Gamma)$, and in this case they act universally elliptic on any $\mathcal{A}$-trees by Lemma \ref{lem: Groves Hull}, or they are conjugates of standard subgroups associated with disconnecting complete graphs of $\Gamma$, as in the $(\mathcal{A,H})$-JSJ decomposition. In this case, there exist some subgraphs $\Gamma'$ of $\Gamma$ such that our disconnecting subgraphs are minimal among complete subgraphs that disconnect $\Gamma'$. By Lemma \ref{lem: complete graphs elliptical} (2), edge stabilisers of $T_\Theta$ act universally elliptic on each $\mathcal{A}$-tree on which $G_{\Gamma'}$ acts, and in particular over any $\mathcal{A}$-tree on which $G$ acts. This shows that $(T_\Theta,G)$ is universally elliptic.

In order to prove that $(T_\Theta,G)$ dominates any other $(\mathcal{A})$-tree $(T',G)$, we need to prove that the action on $T'$ of a vertex stabiliser $H$ of $T_\Theta$ is elliptic for each $(T',G)$ universally elliptic $\mathcal{A}$-tree. Up to conjugation, we can assume that $H$ is a standard subgroup that corresponds to a vertex group of $(\mathcal{G}_\Theta, \Theta)$. If $H$ is non-abelian, then it is a standard subgroup associated with a subgraph without disconnecting complete graphs and its action on $T'$ is elliptic by Theorem \ref{thm: abelian splitting of RAAGs}.
Assume now that $H$ is abelian and suppose that there exists a canonical generator $v$ in $H$ such that its action on $T'$ is hyperbolic. We next show that $v$ is a hanging vertex. By Lemma \ref{lem: Groves Hull}, $\mbox{Star}(v)$ must be a complete graph. Let $w\in \link(v)$ such that $w$ is an element of $H$. Since $H$ is abelian and  $\langle v\rangle\cong \widehat\Z_\C$,  $H$ can not be virtually procyclic and so by \cite[Theorem 7.1.7]{RibesBook}, there must be a nontrivial element $g\in \langle v, w \rangle$ contained in an edge stabiliser of $T'$. As $(T',G)$ is a universally elliptic $\mathcal{A}$-tree, $g$ must be a universally elliptic element and $w\in \alpha(g)$. By Remark \ref{rem: hyperbolic splitting}, the action of $g$ on the standard pro-$\C$ tree of the pro-$\C$ HNN extension
$$G=HNN(G_{\Gamma \smallsetminus \{w\}}, G_{\tiny{\link(w)}}, id)$$
is hyperbolic, and since $g$ is universally elliptic on $\mathcal A$-trees, this implies that $G_{\tiny{\link(w)}}$ is not abelian and so $\mbox{Star}(w)$ is not abelian either. As this is true for each $w\in \link(v)$, we conclude that $v$ is a hanging vertex. However, by the construction of the decomposition, hanging vertices are not contained in any vertex stabilisers of $T_\Theta$ and so we arrived at a contradiction. This proves that the action on $T'$ of each canonical generator in $H$ is elliptic and, applying Lemma \ref{lem: complete graphs elliptical} (1), we conclude that $H$ is elliptic for its action on $T'$, as desired.

This proves that $(T_\Theta, G)$ is an $\mathcal{A}$-JSJ decomposition of $G$.
\end{proof}

We provide an example of an $(\mathcal{A,H})$-JSJ decomposition and an $\mathcal{A}$-JSJ decomposition of a pro-$\C$ RAAG.

\begin{example}
Consider the pro-$\C$ RAAG associated with the graph $P_4$, which is

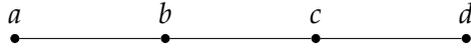
\begin{center}
\begin{tikzpicture}[Ahmadi/.style={circle,fill,draw,inner sep=0pt,minimum size=3pt}]
\node [Ahmadi](a) at (0,0) [label=above:$a$]{};
\node [Ahmadi](b) at (2,0) [label=above:$b$]{};
\node [Ahmadi](c) at (4,0) [label=above:$c$]{};
\node [Ahmadi](d) at (6,0) [label=above:$d$]{};

\path [>=latex,shorten >=0.1]
(a) edge (b)
(b) edge (c)
(c) edge (d)
;
\end{tikzpicture}    
\end{center}
As the $(\mathcal{A,H})$ and $\mathcal{A}$-JSJ decompositions are uniquely determined by the associated graph of groups, we describe only this graph of groups, writing edge and vertex groups next to the corresponding edge and vertex.
A minimal disconnecting complete graph in $P_4$ is $b$. The subgroup $\langle a,b \rangle$ is abelian, whereas $c$ is a disconnecting complete subgraph of the graph generated by $b,c,d$.
The graph of groups decomposition of $G_{P_4}$

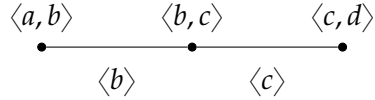
\begin{center}
\begin{tikzpicture}[Ahmadi/.style={circle,fill,draw,inner sep=0pt,minimum size=3pt}]
\node [Ahmadi](a) at (0,0) [label=above: {$\langle a,b \rangle$}]{};
\node [Ahmadi](b) at (2,0) [label=above:{$\langle b,c \rangle$}]{};
\node [Ahmadi](c) at (4,0) [label=above:{$\langle c,d \rangle$}]{};
\node (e1) at (1,0)[label=below:{$\langle b\rangle$}]{};
\node (e1) at (3,0)[label=below:{$\langle c \rangle$}]{};

\path [>=latex,shorten >=0.1]
(a) edge (b) 
(b) edge (c)
;
\end{tikzpicture}
\end{center}

satisfies the conditions of Theorem \ref{thm: A-H JSJ decomposition} and so the corresponding tree defines an $(\mathcal{A,H})$-JSJ decomposition of $P_4$.

The vertex $a$ and $d$ are hanging vertices, because $\mbox{Star}(a)$, $\mbox{Star}(d)$ are complete graphs and $\mbox{Star}(b)$, $\mbox{Star}(c)$ are not complete. For this reason we substitute each of the vertices corresponding to $\langle a,b \rangle$ and $\langle c,d \rangle$ with a vertex and a loop, both with associated group $G_{\tiny{\link(a)}}=\langle b \rangle$ and $G_{\tiny{\link(d)}}=\langle c \rangle$ respectively.

\begin{center}
\begin{tikzpicture}[Ahmadi/.style={circle,fill,draw,inner sep=0pt,minimum size=3pt}]
\node [Ahmadi](a) at (0,0) [label=above: {$\langle b \rangle$}]{};
\node [Ahmadi](b) at (2,0) [label=above:{$\langle b,c \rangle$}]{};
\node [Ahmadi](c) at (4,0) [label=above:{$\langle c \rangle$}]{};
\node (e1) at (1,0)[label=below:{$\langle b \rangle$}]{};
\node (e1) at (3,0)[label=below:{$\langle c \rangle$}]{};

\path[every loop/.style={looseness=60}] 
(a) edge [in=210,out=150,loop] node [left] {{$\langle b \rangle$}} (a);
\path[every loop/.style={looseness=60}]
(c) edge [in=330,out=30,loop] node [right] {{$\langle c \rangle$}} (c);
\path [>=latex,shorten >=0.1] (a) edge (b);
\path [>=latex,shorten >=0.1] (b) edge (c);
\end{tikzpicture}
\end{center}

This graph of groups is not reduced. After reducing it,
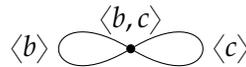
\begin{center}
\begin{tikzpicture}[Ahmadi/.style={circle,fill,draw,inner sep=0pt,minimum size=3pt}]
\node [Ahmadi](b) at (2,0) [label=above:{$\langle b,c \rangle$}]{};

\path[every loop/.style={looseness=60}] 
(b) edge [in=210,out=150,loop] node [left] {$\langle b \rangle$} (b);
\path[every loop/.style={looseness=60}]
(b) edge [in=330,out=30,loop] node [right] {$\langle c \rangle$} (c);
;
\end{tikzpicture}
\end{center}

we obtain a graph of groups decomposition of $P_4$ satisfying the conditions of Theorem \ref{thm: A JSJ decomposition} and so the associated tree is a $\mathcal{A}$-JSJ decomposition of $G_{P_4}$.
\end{example}

\bibliographystyle{alpha}
\bibliography{JSJpropraags}

\begin{thebibliography}{CRdGZ22}

\bibitem[Cha99]{ChatzidakisTorsion}
Zo\'{e} Chatzidakis.
\newblock Torsion in pro-{$p$} completions of torsion-free groups.
\newblock {\em J. Group Theory}, 2(1):65--68, 1999.

\bibitem[Cha07]{Charney}
Ruth Charney.
\newblock An introduction to right-angled {A}rtin groups.
\newblock {\em Geom. Dedicata}, 125:141--158, 2007.

\bibitem[Cla14]{Clay}
Matt Clay.
\newblock When does a right-angled {A}rtin group split over {$\Bbb{Z}$}?
\newblock {\em Internat. J. Algebra Comput.}, 24(6):815--825, 2014.

\bibitem[CRdGZ22]{CRLdG22}
Montserrat Casals-Ruiz and Jone~Lopez de~Gamiz~Zearra.
\newblock On finitely generated normal subgroups of right-angled artin groups
  and graph products of groups, 2022.
\newblock to appear in Proc. AMS.

\bibitem[CZ22a]{CastZal22}
Ilaria Castellano and Pavel Zalesskii.
\newblock A pro-$p$ version of sela's accessibility and poincar\'e duality
  pro-$p$ groups, 2022.

\bibitem[CZ22b]{ChatzidakisZalesskii}
Zo\'{e} Chatzidakis and Pavel Zalesskii.
\newblock Pro-{$p$} groups acting on trees with finitely many maximal vertex
  stabilizers up to conjugation.
\newblock {\em Israel J. Math.}, 247(2):593--634, 2022.

\bibitem[DK92]{DuchampKrob}
G.~Duchamp and D.~Krob.
\newblock The lower central series of the free partially commutative group.
\newblock {\em Semigroup Forum}, 45(3):385--394, 1992.

\bibitem[GH17]{Groves}
Daniel Groves and Michael Hull.
\newblock Abelian splittings of right-angled {A}rtin groups.
\newblock In {\em Hyperbolic geometry and geometric group theory}, volume~73 of
  {\em Adv. Stud. Pure Math.}, pages 159--165. Math. Soc. Japan, Tokyo, 2017.

\bibitem[GL17]{Guirardel}
Vincent Guirardel and Gilbert Levitt.
\newblock J{SJ} decompositions of groups.
\newblock {\em Ast\'{e}risque}, 395:vii+165, 2017.

\bibitem[HR85]{HeZa85}
Wolfgang Herfort and Luis Ribes.
\newblock Torsion elements and centralizers in free products of profinite
  groups.
\newblock {\em J. Reine Angew. Math.}, 358:155--161, 1985.

\bibitem[Kob13]{Koberda}
Thomas Koberda.
\newblock {Right-angled Artin groups and their subgroups}.
\newblock \url{https://users.math.yale.edu/users/koberda/raagcourse.pdf}, 2013.
\newblock [Online; accessed 19-July-2008].

\bibitem[KW16]{Wilkes}
Robert Kropholler and Gareth Wilkes.
\newblock Profinite properties of {RAAG}s and special groups.
\newblock {\em Bull. Lond. Math. Soc.}, 48(6):1001--1007, 2016.

\bibitem[KZ23]{KZ23}
Dessislava Kochloukova and Pavel Zalesskii.
\newblock Finitely generated normal pro-$\mathcal c$ subgroups in right angled
  artin pro-$\mathcal c$ groups, 2023.

\bibitem[Lub93]{Lubotzky}
Alexander Lubotzky.
\newblock Torsion in profinite completions of torsion-free groups.
\newblock {\em Quart. J. Math. Oxford Ser. (2)}, 44(175):327--332, 1993.

\bibitem[Rib17]{RibesBook}
Luis Ribes.
\newblock {\em Profinite graphs and groups}, volume~66 of {\em Ergebnisse der
  Mathematik und ihrer Grenzgebiete. 3. Folge. A Series of Modern Surveys in
  Mathematics [Results in Mathematics and Related Areas. 3rd Series. A Series
  of Modern Surveys in Mathematics]}.
\newblock Springer, Cham, 2017.

\bibitem[RZ00]{NewHorizons}
Luis Ribes and Pavel Zalesskii.
\newblock Pro-{$p$} trees and applications.
\newblock In {\em New horizons in pro-{$p$} groups}, volume 184 of {\em Progr.
  Math.}, pages 75--119. Birkh\"{a}user Boston, Boston, MA, 2000.

\bibitem[RZ10]{RZBook}
Luis Ribes and Pavel Zalesskii.
\newblock {\em Profinite groups}, volume~40 of {\em Ergebnisse der Mathematik
  und ihrer Grenzgebiete. 3. Folge. A Series of Modern Surveys in Mathematics
  [Results in Mathematics and Related Areas. 3rd Series. A Series of Modern
  Surveys in Mathematics]}.
\newblock Springer-Verlag, Berlin, second edition, 2010.

\bibitem[SZ22]{SnopceZalesskii}
Ilir Snopce and Pavel Zalesskii.
\newblock Right-angled {A}rtin pro-{$p$} groups.
\newblock {\em Bull. Lond. Math. Soc.}, 54(5):1904--1922, 2022.

\bibitem[WZ17]{WilZal17}
Henry Wilton and Pavel Zalesskii.
\newblock Distinguishing geometries using finite quotients.
\newblock {\em Geom. Topol.}, 21(1):345--384, 2017.

\bibitem[Zal95]{Zal95}
P.~A. Zalesski\u{\i}.
\newblock Normal divisors of free constructions of pro-finite groups, and the
  congruence kernel in the case of a positive characteristic.
\newblock {\em Izv. Ross. Akad. Nauk Ser. Mat.}, 59(3):59--76, 1995.

\bibitem[ZM88]{Melnikov}
P.~A. Zalesskii and O.~V. Mel'nikov.
\newblock Subgroups of profinite groups acting on trees.
\newblock {\em Mat. Sb. (N.S.)}, 135(177)(4):419--439, 559, 1988.

\end{thebibliography}

\end{document}